\documentclass[letterpaper,12pt,reqno]{amsart}
\usepackage{graphicx,amsmath,amssymb,amsthm}

\usepackage{srcltx}
\usepackage[latin1]{inputenc}
\usepackage[T1]{fontenc}

\usepackage[left=1in,right=1in, top=1in, bottom=1in]{geometry}

\newtheorem{X}{X}[section]
\newtheorem{corollary}[X]{Corollary}

\newtheorem{lemma}[X]{Lemma}
\newtheorem{hypothesis}[X]{Hypothesis}
\newtheorem{proposition}[X]{Proposition}
\newtheorem{theorem}[X]{Theorem}

\newtheorem{innercustomthm}{Hypothesis}
\newenvironment{customthm}[1]
  {\renewcommand\theinnercustomthm{#1}\innercustomthm}
  {\endinnercustomthm}

\theoremstyle{definition}
\newtheorem{remark}[X]{Remark}

\newcommand{\wt}{w\left(\frac dD\right)}

\newcommand{\sumdw}{\underset{\substack{ (d,N_E)=1}}{\sum\nolimits^{*}} w\left( \frac dD \right) }
\newcommand{\sumab}{\underset{\substack{ a,b }}{\sum \nolimits^{*} } w\left(\frac aA,\frac bB \right)}
\newcommand{\sumc}{\underset{\substack{0< |d| \leq D \\ (d,N_E)=1}}{\sum \nolimits^{*} } }

\title[Average Rank]{ A conditional determination of the average rank of elliptic curves}
\author{Daniel Fiorilli}
\address{Department of Mathematics, University of Michigan, 530 Church Street, Ann Arbor MI 48109 USA}
\email{fiorilli@umich.edu}
\date{\today}

\begin{document}

\begin{abstract}

Under a hypothesis which is slightly stronger than the Riemann Hypothesis for elliptic curve $L$-functions, we show that both the average analytic rank and the average algebraic rank of elliptic curves in families of quadratic twists are exactly $\frac 12$. As a corollary we obtain that under this last hypothesis, the Birch and Swinnerton-Dyer Conjecture holds for almost all curves in our family, and that asymptotically one half of these curves have algebraic rank $0$, and the remaining half $1$. We also prove an analogous result in the family of all elliptic curves. A way to interpret our results is to say that nonreal zeros of elliptic curve $L$-functions in a family have a direct influence on the average rank in this family. Results of Katz-Sarnak and of Young constitute a major ingredient in the proofs.
\end{abstract}
\maketitle

\section{Introduction and statement of results}

Let $E$ be an elliptic curve over $\mathbb Q$ whose minimal Weierstrass equation is
\begin{equation}
y^2= x^3+ax+b
\label{equation Weierstrass form}
\end{equation}
with $a,b \in \mathbb Z$ such that $p^4 \mid a \Rightarrow p^6 \nmid b$. It was first conjectured by Goldfeld \cite{Go} that the average analytic rank\footnote{The analytic rank of $E$ is the order of vanishing of $L(E,s)$ (see Section \ref{section prelims}) at $s=1$, and its algebraic rank is the rank of the abelian group of $\mathbb Q$-points on $E$.} over the family of quadratic twists of a fixed curve $E$ should be exactly $\frac 12$. This follows from the widely believed Katz-Sarnak density conjecture, which asserts that this family has \emph{orthogonal} symmetry. The family of all Weierstrass curves \eqref{equation Weierstrass form} is also believed to have orthogonal symmetry, and hence it is believed that the average rank of all elliptic curves ordered by height should also be $\frac 12$. 

In the family of quadratic twists of a fixed elliptic curve, Goldfeld \cite{Go} showed under the Riemann Hypothesis for elliptic curve $L$-functions, which we will denote by ECRH (see below),
that the average analytic rank is at most $3.25$. It was then proved by Brumer \cite{Br} that in the family of all elliptic curves, the average analytic rank is at most $2.3$, again under ECRH. Subsequently, Heath-Brown \cite{HB} improved Goldfeld's upper bound to $\frac 32$, and Brumer's upper bound to $2$. Heath-Brown also showed that the proportion of curves of rank at least $R$ decays exponentially with $R$. Young \cite{Yo} showed under the Grand Riemann Hypothesis that the average analytic rank in the family of all curves is at most $\frac {25}{14}$, and as a corollary he obtained that the Birch and Swinnerton-Dyer Conjecture holds for a positive proportion of curves in his family. This last corollary is obtain from the deep results of Gross-Zagier \cite{GZ}, Kolyvagin \cite{Ko}, and others. Young also showed corresponding results for several other interesting families of elliptic curves. Finally, Bhargava and Shankar \cite{BS} have recently shown unconditionally that the average algebraic rank in the family of all  curves is at most $0.885$.

\medskip

\noindent \textbf{Hypothesis ECRH.} If $E$ is an elliptic curve over $\mathbb Q$, then the Riemann Hypothesis holds for $L(E,s)$, that is all nontrivial zeros of $L(E,s+\tfrac 12)$ lie on the line $\Re(s)=\tfrac 12$.

\medskip

The goal of the current paper is to deduce Goldfeld's Conjecture from an assumption slightly stronger than ECRH, which is motivated by probabilistic arguments inspired by the work of Montgomery. 
This should be compared with the very recent paper of Bhargava, Kane, Lenstra, Poonen and Rains \cite{BKLPR}, in which the authors study a probability distribution on a certain set of short exact sequences and show that well-known conjectures on the average rank would follow from the assertion that short exact sequences arising from elliptic curves follow this distribution.

\subsection{Quadratic twists of a fixed elliptic curve}

We first consider families of quadratic twists, that is we fix $E$ and consider for squarefree $d\neq 0$ the curve
\begin{equation}
E_d: dy^2= x^3+ax+b.
\label{equation quadratic twist}
\end{equation}

For technical reasons we will mostly consider the values of $d$ which are coprime with $N_E$. Our first main result is a conditional proof that the average analytic rank of $E_d$ is exactly $\frac 12$. Here and throughout, $\sum_d^*$ will denote a sum over squarefree integers $d$ and $N(D)$ will denote the number of squarefree integers $0<|d|\leq D$ with $(d,N_E)=1$.
\begin{theorem}
\label{theorem main twists}
Assume ECRH and assume that Hypothesis M below holds for some non-negative Schwartz weight function $w$ with $w(0)>0$. Then the average of $r_{an}(E_d)$, the analytic rank of $L(E_d,s)$, is exactly $\frac 12$:
$$  \lim_{D \rightarrow \infty}   \frac{1}{N(D)} \sumc r_{an}(E_d) = \frac 12.$$
\end{theorem}
Hypothesis M is a statement about nonreal zeros of elliptic curve $L$-functions which is only slightly stronger than ECRH. Theorem \ref{theorem main twists} thus asserts that the imaginary parts of 
these nonreal zeros have a direct influence on the average order of vanishing of $L(E,s)$ at the central point. 
 
We translate Theorem \ref{theorem main twists} into a statement about algebraic ranks using the deep results of Gross-Zagier \cite{GZ} and Kolyvagin \cite{Ko}.
\begin{corollary}
\label{first corollary}
Assume ECRH and Hypothesis M for some non-negative Schwartz function $w$ with $w(0)>0$. Then the Birch and Swinnerton-Dyer Conjecture holds for almost all curves $E_d$ (with $\mu^2(d)=1$, $(d,N_E)=1$), and asymptotically one half of these curves have algebraic rank $0$, and one half have algebraic rank $1$.
\end{corollary}

\begin{remark}
One can adapt the arguments of Theorem \ref{theorem main twists} to show that under similar hypotheses\footnote{The needed hypotheses are ECRH and the statement that \eqref{equation hypothesis} holds without the restriction $(d,N_E)=1$ in the sum over $d$.} we have
$$  \lim_{D \rightarrow \infty}   \frac{1}{N_0(D)} \underset{\substack{0< |d| \leq D }}{\sum \nolimits^{*} }  r_{an}(E_d) = \frac 12,$$
where $N_0(D)$ denotes the number of squarefree integers $0<|d|\leq D$. If $N_E$ is squarefree, then Lemma \ref{lemma equidistribution root number} shows that the root number is equidistributed in the family $\{ E_d : \mu^2(d)=1\}$, and hence asymptotically one half of these curves have algebraic rank $0$, and one half have algebraic rank $1$.
\end{remark}

Our working hypothesis is slightly stronger than ECRH, but only by one logarithm. We will see that Montgomery's probabilistic arguments predict a much stronger estimate. Here and throughout, $w(t)$ will denote a fixed non-negative Schwartz\footnote{Hence $w(t)$ is smooth and rapidly decaying.} test function such that $w(0)> 0$; this weight function will facilitate the analysis.

\medskip
\noindent \textbf{Hypothesis M.}
\textit{There exists $0<\delta <1$ such that in the range\footnote{It is actually sufficient to assume that \eqref{equation hypothesis} holds for the specific value $x=D^{2-\delta}$ only.} $  D^{2-\delta} \leq x \leq  2D^{2-\delta}$ we have the following estimate: }
\begin{equation}
\sumdw  \sum_{\rho_d \notin \mathbb R} \frac{x^{\rho_d }}{\rho_d(\rho_d+1)}  = o(Dx^{\frac 12}), 
\label{equation hypothesis}
\end{equation} 
\textit{where $\rho_d$ runs over the nontrivial zeros of $L(E_d,s+\frac 12)$.}
\medskip

\begin{remark}
The Riemann Hypothesis for $L(E_d,s)$ implies that the left hand side of \eqref{equation hypothesis} is $O(Dx^{\frac 12}\log (DN_E))$, hence Hypothesis M is stronger, but only by one logarithm.
A probabilistic argument
which will be sketched in Appendix \ref{appendix} suggests the stronger bound
\begin{equation}
\label{equation strong montgomery}
\sumdw \sum_{\rho_d \notin \mathbb R} \frac{x^{\rho_d }}{\rho_d(\rho_d+1)}  = O_{\epsilon,E}(D^{\frac 12+\epsilon} x^{\frac 12}),
\end{equation}  
when $x$ is large enough in terms of $D$. This is based on Montgomery's Conjecture on primes in arithmetic progressions.
We will see in Corollary \ref{corollary curves of high rank} that such an estimate implies a quantitative bound for the number of elliptic curves of rank $\geq 2$. It will actually be sufficient to assume the  weaker Hypothesis M$(\delta,\eta)$ (see below), for some $0<\delta<1$ and $0<\eta<\tfrac 12$. 
\end{remark}

\begin{remark}
An important fact used in Appendix \ref{appendix} to conjecture \eqref{equation strong montgomery} as well as \eqref{equation hypothesis all curves} is the fact that the families of elliptic curves we are considering contain at most a bounded number of isogenous elliptic curves. Indeed we are considering families of minimal\footnote{$E_d$ can be rewritten as $y^2=x^3+d^2ax+d^3b$, which is minimal for $d\neq 0$ squarefree.} Weierstrass equations, which implies that no pairs of curves in these families are isomorphic. Moreover, a Theorem of Mazur states that at most a bounded number of such elliptic curves can be isogenous, and it follows from the Isogeny 
Theorem (see Lemma \ref{lemma bounded number of repetitions}) that at most a bounded number of elliptic curves in this family have the same $L$-function.

\end{remark}

\begin{remark}
In \eqref{equation hypothesis}, one can replace $1/\rho_d(\rho_d +1)$ by the Mellin Transform of any function satisfying appropriate decay conditions (see Property D in Section \ref{section prime number theorems}) evaluated at $\rho_d$, and the same results will follow (see the proof of Lemmas \ref{lemma Sym2} and \ref{lemma Sym3}, and Propositions \ref{proposition average rank} and \ref{proposition KS}). We made the specific choice $h(x)=\max(1-x,0)$ in \eqref{equation hypothesis} for simplicity, keeping in mind that \eqref{equation hypothesis} is more likely to hold if the Mellin Transform of $h$ decays on vertical lines.
\end{remark}

A weaker version of \eqref{equation strong montgomery} is the following.

\medskip
\noindent \textbf{Hypothesis M($\delta$,$\eta$).}
\textit{In the range $  D^{2-\delta} \leq x \leq  2D^{2-\delta}$ we have the following estimate: }
\begin{equation}
\sumdw  \sum_{\rho_d \notin \mathbb R} \frac{x^{\rho_d }}{\rho_d(\rho_d+1)}  = O_{E}(D^{1-\eta}x^{\frac 12}), 
\label{equation hypothesis M eta}
\end{equation} 
\textit{where $\rho_d$ runs over the nontrivial zeros of $L(E_d,s+\frac 12)$ and the constant implied in the error term might depend on $\delta$ and $\eta$.}
\medskip

We now show that the sharper bound \eqref{equation hypothesis M eta} implies a more precise estimate for the average rank. We will assume the Riemann Hypothesis for some symmetric power $L$-functions of $E_d$.

\begin{customthm}{ECRH$^+$}\label{ECRH+}
The Riemann hypothesis holds for $\zeta(s)$, and holds for $L(\text{Sym}^k E,s)$ for every elliptic curve $E$ over $\mathbb Q$ and $1\leq k \leq 3$. 
\end{customthm}

\begin{theorem}
\label{theorem main error term}
Assume ECRH$^+$ and Hypothesis M$(\delta,\eta)$ for some $0<\delta<1$ and $0<\eta<\frac 12$, for some non-negative Schwartz weight $w$ with $w(0)>0$. Then for any fixed $\epsilon>0$ we have
\begin{equation}
\frac 1{N(D)}\sumc r_{\text{an}}(E_d) = \frac 12 +O_{\epsilon,E}( D^{1-\frac{\delta}4+\epsilon}+D^{1-\eta}),
\label{equation main theorem error term}
\end{equation}
where $D^*$ denotes the number of squarefree integers in the interval $[1,D]$.
\end{theorem}

\begin{corollary}
Under ECRH$^+$ and Hypothesis M$(\delta,\eta)$ for some $0<\delta<1$ and $0<\eta<\frac 12$, we have the following bound for the number of curves $E_d$ of rank $\geq 2$:
$$\underset{\substack{0< |d| \leq D \\ r_{\text{al}}(E_d) \geq 2 \\ (d,N_E)=1}}{\sum \nolimits^{*} }  1 \ll_{\epsilon,E}D^{1-\min(\tfrac{\delta}4,\eta)+\epsilon},\hspace{.5cm} \underset{\substack{0< |d| \leq D \\ r_{\text{an}}(E_d) \geq 2 \\ (d,N_E)=1}}{\sum \nolimits^{*} }  1 \ll_{\epsilon,E}D^{1-\min(\tfrac{\delta}4,\eta)+\epsilon}.   $$
In particular, the proportion of elliptic curves $E_d$ with $0<|d|\leq D $ squarefree for which the Birch and Swinnerton-Dyer Conjecture does not hold is $\ll D^{-\min(\tfrac{\delta}4,\eta)+o(1)} $.
\label{corollary curves of high rank}
\end{corollary}

\begin{remark}
This should be compared with the following conjecture of Sarnak:
$$ V_E(D):=\underset{\substack{0< |d| \leq D \\ \epsilon(E) \chi_d(-N_E)=1 \\    r_{\text{an}}(E_d) \geq 2}}{\sum \nolimits^{*} }  1 \approx D^{\frac 34}, $$
where $\epsilon(E)$ denotes the root number of $L(E,s)$. Using random matrix theory, Conrey, Keating, Rubinstein and Snaith \cite{CKRS} refined this conjecture to $ V_E(D) \sim c_E D^{\frac 34} (\log D)^{b_E}. $ Here, $b_E$ can be made explicit \cite{DW}. Interestingly, if we set $\eta=\frac 12-\epsilon$ and $\delta=1-\epsilon$ in Hypothesis M$(\delta,\eta)$, which we believe is best possible choice of $\eta$ for which \eqref{equation hypothesis M eta} holds (this corresponds to Montgomery's Conjecture for primes in arithmetic progressions, see Appendix \ref{appendix}), then in Corollary \ref{corollary curves of high rank} we recover the upper bound in Sarnak's Conjecture: $V_E(D)\ll_{\epsilon} D^{\frac 34+\epsilon} $. Note also that if we wish to take $\delta>1$ in Hypothesis M$(\delta,\eta)$, then we have to add the error term $O_{\epsilon}(D^{\frac 12+ \frac{\delta}4+\epsilon})$ in \eqref{equation main theorem error term} (put $P=D^{2-\delta}$ in \eqref{equation averange rank GRH+}), and hence the resulting bound on the number of curves of rank $\geq 2$ can never be better than $O_{\epsilon}(D^{\frac 34+\epsilon})$.
\end{remark}

Corollary \ref{corollary curves of high rank} also has an implication on  the average algebraic rank. The implication is not direct, since one needs a bound on the algebraic rank of elliptic curves in families of quadratic twists (see Lemma \ref{lemma bound algebraic rank}). 

\begin{theorem}
\label{theorem average algebraic rank}
Assume ECRH and assume that there exists $0<\delta<1$ such that in the range $D^{2-\delta}\leq x \leq 2D^{2-\delta}$, the left hand side of \eqref{equation hypothesis} is $o(D x^{\frac 12}/\log\log D)$, for some non-negative Schwartz $w$ with $w(0)>0$. Then the average algebraic rank of $E_d$ is exactly $\tfrac 12$:
$$\lim_{D \rightarrow \infty} \frac 1{N(D)}\sumc  r_{\text{al}}(E_d)=\frac 12.  $$
\end{theorem}

\begin{remark}
\label{remark KS}
Sarnak noticed\footnote{Private conversation.} that Montgomery's Conjecture for primes in arithmetic progressions implies the Katz-Sarnak prediction for the $1$-level density in the family of Dirichlet $L$-functions modulo $q$. In making the analogous conjecture in families of elliptic curve $L$-functions, one runs into the problem that if not excluded, zeros at the central point could give significant contributions to the sum on the left hand side of \eqref{equation hypothesis}. In contrast with Dirichlet $L$-functions, many elliptic curve $L$-functions do vanish at the central point, and thus the direct analogue of Montgomery's Conjecture is not expected to hold. To fix this problem, we excluded the real zeros in \eqref{equation hypothesis}; however the Katz-Sarnak prediction does not follow directly from \eqref{equation hypothesis}, because of the absence of real zeros. Notice also that the range $x\asymp D^{2-\delta}$ in Hypothesis M corresponds in the Katz-Sarnak problem to test functions whose Fourier transform have small support, 
and this is not sufficient in order to deduce Goldfeld's Conjecture from an estimate on the $1$-level density.
\end{remark}

\subsection{The family of all elliptic curves}

We also give an analogue of Theorem \ref{theorem main twists} in the family of all elliptic curves. 
Here $w(t_1,t_2)$ will denote a fixed Schwartz test function on $\mathbb R^2$. 

\begin{hypothesis}
\label{hypothesis all curves}
There exists $\delta>0$ such that in the range $ X^{\frac {7}9-\delta} \leq x \leq 2X^{\frac {7}9-\delta}$ we have 
\begin{equation}
 \sumab \sum_{\rho_{a,b} \notin \mathbb R} x^{\rho_{a,b} }\Gamma(\rho_{a,b})  = o(AB x^{\frac 12}), 
\label{equation hypothesis all curves}
\end{equation} 
where $A=X^{\frac 13}$, $B= X^{\frac 12}$ and $\rho_{a,b}$ runs through the zeros of $L(E_{a,b},s)$, the $L$-function of the elliptic curve $E_{a,b} : y^2=x^3+ax+b$. The star on the sum over $(a,b)$ means that we are summing over the couples for which $p^4 \mid a \Rightarrow p^6 \nmid b$ (in particular $b\neq 0$).

\end{hypothesis}

\begin{remark}
As is the case with Hypothesis M, Hypothesis \ref{hypothesis all curves} is stronger than ECRH by only one logarithm, since ECRH implies the bound
\begin{equation}
 \underset{\substack{ a,b }}{\sum \nolimits^{*} } w\left(\frac aA,\frac bB \right) \sum_{\rho_{a,b} \notin \mathbb R} x^{\rho_{a,b} }\Gamma(\rho_{a,b})  \ll AB x^{\frac 12} \log(ABN_E). 
\end{equation} 
Hypothesis \ref{hypothesis all curves} is motivated by probabilistic considerations similar to those presented in Appendix \ref{appendix}. An important requirement is that we only consider minimal Weierstrass Equations, since no two curves in this family are isomorphic\footnote{It follows from \cite[Section III]{Sil} that two elliptic curves in Weierstrass Form $E: y^2=x^3+ax+b$ and $E': y^2=x^3+a'x+b'$ are isomorphic over $\mathbb Q$ if and only if $a'=u^4a$ and $b'=u^6b$ for some $u\in \mathbb Q^{\times}$.}, and by Lemma \ref{lemma bounded number of repetitions}, at most a bounded number of curves in this family have matching $L$-functions. It is then natural to conjecture that distinct $L$-functions (in this case $L$-functions of two non-isogenous elliptic curves) have distinct nonreal zeros.
\end{remark}

Assuming Hypothesis \ref{hypothesis all curves}, we will show using Young's results \cite{Yo} that the average analytic rank of the elliptic curve $ E_{a,b} : y^2 = x^3+ ax+b $
is exactly $\frac 12$.
\begin{theorem}
\label{theorem all curves}
Assume ECRH, and assume that Hypothesis \ref{hypothesis all curves} holds for some non-negative Schwartz weight $w$. Then we have that
$$ \lim_{\substack{ A,B \rightarrow \infty \\ A^3=B^2 }} \frac 1{W(A,B)}\underset{\substack{ a,b }}{\sum \nolimits^{*} } w\left( \frac aA,\frac bB \right) r_{an}(E_{a,b})  = \frac 12, $$
where the sum is taken over the pairs $(a,b)$ for which $p^4 \mid a \Rightarrow p^6 \nmid b$, and $W(A,B)$ denotes the sum of $w\left( \frac aA,\frac bB \right)$ over all such pairs.
\end{theorem}
If one also assumes the widely believed conjecture that the root number is equidistributed in the family of all elliptic curves, then the analogue of Corollary \ref{first corollary} follows.

\section{Preliminaries and Prime Number Theorems}
\label{section prelims}

Fix an elliptic curve in Weierstrass form
$$ E: y^2=x^3+ax+b$$
with $a,b \in \mathbb Z$, discriminant $\Delta_E=-16(4a^3+27b^2)$ and conductor $N_E$. For $p\nmid N_E$, consider the trace of the Frobenius automorphism $a_p(E)=p+1-\#  E_p(\mathbb F_p)$, which satisfies Hasse's bound $|a_p(E)|\leq 2\sqrt p$. Here, $ E_p$ denotes the reduction of $E$ modulo $p$. We extend the definition of $a_p(E)$ to all primes by setting
$$ a_p(E) := \begin{cases}
1 \text{ if } E \text{ has split multiplicative reduction at }p \\
-1 \text{ if } E \text{ has nonsplit multiplicative reduction at }p \\
0 \text{ if } E \text{ has additive reduction at }p, \\
\end{cases}$$ 
and define the $L$-function of $E$ as follows:
$$ L(E,s):= \prod_{p\mid N_E} \left( 1-\frac {a_p(E)}{p^s}\right)^{-1} \prod_{p\nmid N_E} \left( 1-\frac{a_p(E)}{p^s}+\frac p{p^{2s}}\right)^{-1}. $$
One can see that the formula $a_p(E)=p+1-\# E_p(\mathbb F_p)$ still holds for primes dividing $N_E$ \cite[Section X.2]{Kn}.
It follows from the groundbreaking work of Wiles \cite{W}, Taylor and Wiles \cite{TW}, and Breuil, Conrad, Diamond, and Taylor \cite{BCDT},  that $L(E,s+\frac 12)=L(s,f_E)$ for some cuspidal self-dual newform $f_E$ of weight $2$ and level $N_E$. The gamma factor of $L(E,s+\frac 12)$ is given by $\gamma(f_E,s)= \pi^{-s} \Gamma(\tfrac{s}2+ \tfrac 14) \Gamma(\tfrac{s}2+ \tfrac 34)$, and the completed $L$-function 
$$\Lambda(f_E,s):= N_E^{\frac s2} \gamma(f_E,s) L(s,f_E) $$
satisfies the functional equation $\Lambda(f_E,s)=\epsilon(E)\Lambda(f_E,1-s)$. The trivial zeros of $L(f_E,s)$ are simple zeros at the points $s=-\tfrac 12,-\tfrac 32,-\tfrac 52,... $
One can write 
$$ L(s,f_E)=L(E,s+\tfrac 12)=\prod_{p} \left( 1-\frac{\alpha_p(E)}{p^s}\right)^{-1}\left( 1-\frac{\beta_p(E)}{p^s}\right)^{-1}, $$
where for $p\nmid N_E$ we have $|\alpha_p(E)| =|\beta_p(E)|=1$, $\beta_p(E) = \overline{\alpha_p(E)}$ and $a_p(E)/\sqrt p = \alpha_p(E)+\beta_p(E)$, and for $p\mid N_E$ we have $\alpha_p(E) = a_p(E)/\sqrt p$ and $\beta_p(E)=0$. The symmetric $k$-th power $L$-function of $E$ is then defined as
$$ L(s,\text{Sym}^k f_E)=L(\text{Sym}^k E,s+\tfrac k2) = \prod_p \prod_{j=0}^k \left(1- \frac{\alpha_p(E)^j \beta_p(E)^{k-j}}{p^{s}}\right)^{-1}. $$

In the case $k=2$, it was proven by Shimura \cite{Sh} that $L(s,\text{Sym}^2 f_E)$ can be analytically continued to an entire function of $s$. The gamma factor is given by $\gamma(\text{Sym}^2f_E,s)= \pi^{-\frac{3s}2} \Gamma(\tfrac{s+1}2)^2 \Gamma(\tfrac{s}2+1) $, and the completed $L$-function $$\Lambda(\text{Sym}^2f_E,s):= q(\text{Sym}^2 f)^{\frac s2}\gamma(\text{Sym}^2f_E,s)L(s,\text{Sym}^2 f_E), $$
where the conductor $q(\text{Sym}^2 f_E) = N_E^2$,
satisfies the functional equation $\Lambda(\text{Sym}^2f_E,s)=\epsilon(\text{Sym}^2 E)\Lambda(\text{Sym}^2f_E,1-s)$ (see \cite[Section 5.12]{IK} and \cite{CM}). 

The Langlands Program predicts every symmetric power $L(\text{Sym}^k E,s)$ to be the $L$-function of a self-dual GL$(k+1)$ cuspidal automorphic form. In particular, it would follow that these $L$-functions have a analytic continuation to the whole $s$-plane (except for a possible pole at $s=1$), have a functional equation, and are of order $1$ (see \cite[Theorem 5.41]{IK}). A precise prediction of the conductor, root number and Gamma Factors of $L(\text{Sym}^k E,s)$ is given in \cite{CM}. The Riemann Hypothesis for $L(\text{Sym}^k E,s+\tfrac k2) = L(s,\text{Sym}^k f_E)$ states that all nontrivial zeros of this function have real part $\tfrac 12$.

The automorphy of $L(\text{Sym}^k E,s)$ is currently known for $ 1 \leq k \leq 4$ by the work of Gelbart and Jacquet \cite{GJ}, Kim and Shahidi \cite{KiS1,KiS2} and Kim \cite{Ki}. In particular the Rankin-Selberg convolution $L(\text{Sym}^k E\otimes \text{Sym}^k E,s)$ exists, and thus we have a zero-free region \cite[Theorems 5.10 and 5.42]{IK}, from which we will deduce a Prime Number Theorem. This will be essential, as the point of the present paper is to show how the average rank of elliptic curves is determined by the analytical properties of $L(\text{Sym}^2E,s)$; we believe this to hold in greater generality. Note that our error terms will be effective as Goldfeld, Hoffstein and Lieman have shown the non-existence of Landau-Siegel zeros for $L(\text{Sym}^2E,s)$.

\begin{remark}
\label{remark CM}
One has to be careful with possible poles of $L(s,\text{Sym}^kf_E)/L'(s,\text{Sym}^kf_E)$ at $s=1$. If $E$ has no complex multiplication, then the work of Taylor, Clozel, Harris and Shepherd-Barron \cite{Tay,CHT,HST} implies that $L(s,\text{Sym}^kf_E)$ is holomorphic and nonzero at $s=1$. In the CM case however, things are slightly different. Indeed at $s=1$, the function $L(s,\text{Sym}^kf_E)$ is holomorphic and nonzero for $k\equiv 1,2,3 \bmod 4$, but has a simple pole when $4\mid k$. 
\end{remark}

The prime number theorems to appear in this section will be weighted by certain function, in order to obtain absolutely convergent sums over zeros. In the following we assume that $h:(0,\infty) \rightarrow \mathbb R$ is such that its Mellin Transform
$$ \mathcal M h(s) :=\int_0^{\infty} x^{s-1} h(x)dx $$
converges for $\Re(s)\geq 0$, and such that the following property holds.

\medskip
\noindent \textbf{Property D.}
\begin{itemize}
\vspace{-.1cm}
\item For any $N\geq 1$ and in the range $x\geq 1,$ $h(x)\ll_N \frac 1{x^N}.$

\item $\mathcal M h(s)$ can be analytically continued to a meromorphic function with possible poles of order at most one at the points $s=0,-1,-2,...$ 

\item Uniformly for $|\sigma| \leq 1$  and $|t|\geq 1,$  $\mathcal M h(\sigma+it) \ll \frac 1{t^2}$.
\end{itemize}
\label{section prime number theorems}

These properties hold for the typical examples $h(x)=\max(1-x,0)^k$ ($k\geq 1$), and $h(x)=e^{-x}$. In fact they hold when $h$ is any real Schwartz function on $(0,\infty)$. 

\begin{lemma}
\label{lemma Sym2}
Fix an elliptic curve $E$, and assume that $h$ is a function satisfying Property D. Then there exists an effective absolute constant $c>0$ such that
\begin{equation}
\sum_{p}  (\alpha_p(E)^2 + \beta_p(E)^2)  h(p/x) \log p = -x \mathcal M h(1) +O(x e^{-c \log x / (\sqrt{\log x} + \log N_E)}(\log N_E)^2).
\label{Sym2 unconditional}
\end{equation}  
If we assume the Riemann Hypothesis for both $L(\text{Sym}^2 E,s+1)$ and $\zeta(s)$, then
\begin{equation}
\sum_{p}  (\alpha_p(E)^2 + \beta_p(E)^2)  h(p/x) \log p
 =  -x \mathcal M h(1) +O(x^{\frac 12} \log N_E).
\label{Sym2 with Riemann}
\end{equation} 
The constants implied in the error terms depend on $h$.
\end{lemma}

The estimate \eqref{Sym2 with Riemann} will allow us to predict a quantitative bound for the number of elliptic curves of rank $\geq 2$. The following Lemma, which is an application of the automorphy of $L(s,\text{Sym}^3 f_E)$, will strengthen this quantitative bound.

\begin{lemma}
\label{lemma Sym3}
Fix an elliptic curve $E$ and assume that $h$ is a function satisfying Property D, and assume the Riemann Hypothesis for both $L(\text{Sym}^3 E,s+\tfrac 32)$ and
$L(E,s+\frac 12)$. Then we have the bound
\begin{equation}
\sum_{p}  (\alpha_p(E)^3 + \beta_p(E)^3)  h(p/x)\log p \ll x^{\frac 12} \log N_E.
\label{Sym3 with Riemann}
\end{equation} 
\end{lemma}

\begin{proof}[Proof of Lemma \ref{lemma Sym2}]

It follows from the above discussion that the function $L(\text{Sym}^2 E,s+1)=L(s,\text{Sym}^2 f_E)$ is a degree-$3$ entire $L$-function in the sense of Iwaniec and Kowalski \cite[Chapter 5]{IK}.
The logarithmic derivative of this $L$-function is given by 
$$ -\frac{L'(\text{Sym}^2 E,s+1)}{L(\text{Sym}^2 E,s+1)} = \sum_{k\geq 1} \sum_p \frac{(\alpha_p(E)^{2k} +\alpha_p(E)^k \beta_p(E)^k + \beta_p(E)^{2k}) \log p}{p^{ks}}.$$
As noted in Remark \ref{remark CM}, this function is holomorphic at $s=1$. Applying Mellin Inversion we obtain 

\begin{equation}
\sum_{\substack{p \\ k\geq 1}} (\alpha_p(E)^{2k} +\alpha_p(E)^k \beta_p(E)^k + \beta_p(E)^{2k}) h(p^{k}/x) \log p = \frac{-1}{2\pi i}\int_{\Re(s)=2} \frac{L'(s,\text{Sym}^2 f_E)}{L(s,\text{Sym}^2 f_E)} x^s \mathcal M h(s)ds.
\label{equation Perron}
\end{equation}

We have seen that the functional equation for $L(s,\text{Sym}^2 f_E,s)$ takes the form $\Lambda(\text{Sym}^2 f_E,s)=\epsilon(\text{Sym}^2 E)\Lambda(\text{Sym}^2 f_E,1-s)$, that is
\begin{multline*} (\pi^3q(\text{Sym}^2 f_E))^{-\frac s2} \Gamma(\tfrac{s+1}2)^2 \Gamma(\tfrac{s}2+1) L(s,\text{Sym}^2 f_E)\\=\epsilon(\text{Sym}^2 E)(\pi^3q(\text{Sym}^2 f_E))^{-\frac {1-s}2} \Gamma(1-\tfrac{s}2)^2 \Gamma(\tfrac{3-s}2) L(1-s,\text{Sym}^2 f_E).
\end{multline*}

This implies that the trivial zeros of $L(s,\text{Sym}^2 f_E,s)$ are simple zeros at $s=-2m$ ($m \geq 1$) and double zeros at $s=1-2n$ ($n\geq 1$). Pulling the contour of integration to the left in \eqref{equation Perron} until $\Re(s)=-\frac 13$ gives the following estimate (note that $L'(s,\text{Sym}^2 f_E)/L(s,\text{Sym}^2 f_E)$ is holomorphic at $s=1$):
\begin{multline} \sum_{\substack{p \\ k\geq 1}} (\alpha_p(E)^{2k} +\alpha_p(E)^k \beta_p(E)^k + \beta_p(E)^{2k}) h( p^{k}/x) \log p= -\sum_{\rho_{E,2}} x^{\rho_{E,2}}\mathcal M h(\rho_{E,2}) \\ -c_h\frac{L'(0,\text{Sym}^2 f_E)}{L(0,\text{Sym}^2 f_E)}+O\left( x^{-\frac 13}\int_{\Re(s)=-\frac 13} \left|\frac{L'(s,\text{Sym}^2 f_E)}{L(s,\text{Sym}^2 f_E)} \right| \frac{|ds|}{|s|^2}\right),
\label{equation explicit formula Sym2}
\end{multline}
where $\rho_{E,2}$ runs through the nontrivial zeros of $L(s,\text{Sym}^2 f_E)$ with multiplicity, and $c_h$ denotes the residue at $s=0$ of $\mathcal M h(s)$, which might equal zero. To bound the terms on the right hand side of \eqref{equation explicit formula Sym2} we use the functional equation in the form 
$$ \frac{L'(s,\text{Sym}^2 f_E)}{L(s,\text{Sym}^2 f_E)} + \frac{\gamma'(\text{Sym}^2 f_E,s)}{\gamma(\text{Sym}^2 f_E,s)} = \log q(\text{Sym}^2 f_E) + \frac{L'(1-s,\text{Sym}^2 f_E)}{L(1-s,\text{Sym}^2 f_E)} + \frac{\gamma'(\text{Sym}^2 f_E,1-s)}{\gamma(\text{Sym}^2 f_E,1-s)} . $$

On the line $\Re(s)=4/3$ the logarithmic derivative of $L(s,\text{Sym}^2 f_E)$ is bounded by an absolute constant. It follows from the functional equation and the asymptotic properties of the Digamma function\footnote{The function $\psi(z):= \Gamma'(z)/\Gamma(z)$ satisfies $\psi(z)\sim \log z$, in the sector $\{ z \in \mathbb C : |\text{arg}(z-1)|< \pi-\delta  \}$, for $\delta>0$ fixed.}
 that
$$\frac{L'(-\tfrac 13+it,\text{Sym}^2 f_E)}{L(-\tfrac 13+it,\text{Sym}^2 f_E)} \ll \log (N_E|t|), $$
and thus the error term in \eqref{equation explicit formula Sym2} is at most a constant (which can depend on $h$) times $x^{-\frac 13} \log N_E$. 

We first assume the Riemann Hypothesis for $L(s,\text{Sym}^2 f_E)$ and $\zeta(s)$, and we show the conditional result \eqref{Sym2 with Riemann}. Combining \cite[(5.28)]{IK} and \cite[Theorem 5.33]{IK}, we obtain the bound\footnote{Note that $\mathcal M h(s)$ might have a pole at $s=0$, but each nontrivial zero is at a positive distance away from this point.}
\begin{equation}
\label{equation bound logarithmic derivative}
\frac{L'(1,\text{Sym}^2 f_E)}{L(1,\text{Sym}^2 f_E)} \ll \sum_{|\rho_{E,2}| <1}\frac 1{|\rho_{E,2}|} +\log q(\text{Sym}^2 f_E) \leq \sum_{|\rho_{E,2}| <1}2 +\log q(\text{Sym}^2 f_E) \ll \log N_E.  
\end{equation}
It then follows from the functional equation that
$$\frac{L'(0,\text{Sym}^2 f_E)}{L(0,\text{Sym}^2 f_E)} \ll \log N_E.  $$

As for the sum over nontrivial zeros in \eqref{equation explicit formula Sym2}, we apply the Riemann-von Mangoldt Formula \cite[(5.33)]{IK} to obtain the bound
$$\sum_{\rho_{E,2}} x^{\frac 12 + i\gamma_{E,2}}\mathcal Mh(\rho_{E,2}) \ll x^{\frac 12} \log  q(\text{Sym}^2 f_E) \ll  x^{\frac 12} \log N_E.  $$
Combining these bounds we obtain that
$$  \sum_{\substack{p\\ k\geq 1}} (\alpha_p(E)^{2k} +\alpha_p(E)^k \beta_p(E)^k + \beta_p(E)^{2k})  h(p^{k}/x) \log p\ll x^{\frac 12} \log q(\text{Sym}^2 f_E) \ll x^{\frac 12}\log N_E.$$
Note that $|\alpha_p(E)|,|\beta_p(E)| \leq 1$, and for $p\nmid N_E$, $\alpha_p(E)\beta_p(E)=1$. We therefore have 
\begin{align}
\sum_{\substack{p \\ k\geq 1}} &(\alpha_p(E)^{2k} +\alpha_p(E)^k \beta_p(E)^k + \beta_p(E)^{2k})  h(p^{k}/x) \log p \notag \\&=  \sum_{\substack{p }} (\alpha_p(E)^2 +1 + \beta_p(E)^{2})  h(p/x)\log p + O\Big( \sum_{\substack{ p^k \leq x \\ k\geq 2}} \log p + \sum_{p\mid N_E} \log p \Big) \label{equation playing with sym2} \\
&= x \mathcal M h(1)+\sum_{\substack{p }} (\alpha_p(E)^2 + \beta_p(E)^{2})  h(p/x)\log p +O( x^{\frac 12} + \log N_E),
\notag
\end{align} 
by the Riemann Hypothesis for $\zeta(s)$. The proof follows.

We now prove the unconditional result \eqref{Sym2 unconditional}, using the zero-free region of \cite[Theorem 5.44]{IK}. 
If $E$ has no CM, then $f_E$ is not the lift of a $GL(1)$ $L$-function and the work of Goldfeld, Hoffstein and Lieman \cite{GHL} implies the nonexistence of Landau-Siegel zeros for the associated $L$-function. It follows from \cite[Theorem 5.44]{IK} that for an absolute effective constant $c>0$ we have, without exception, the following bounds:
$$\frac c{\log(N_E (|\Im(\rho_{E,2})|+3))}< \Re(\rho_{E,2}) < 1-\frac c{\log(N_E (|\Im(\rho_{E,2})|+3))}. $$ 
These bounds also hold in the CM case, since in this situation the associated Hecke Grossencharacter $\xi_{E/K}$ is complex, and thus Landau-Siegel zeros do not exist.

Using this zero-free region and introducing a parameter $T\geq 1$, the sum over $\rho_{E,2}$ on the right hand side of \eqref{equation explicit formula Sym2} is at most (note that $\mathcal M h(s)$ might have a simple pole at $s=0$)
\begin{align*}
\sum_{|\Im(\rho_{E,2})|>T} \frac{x}{|\rho_{E,2}|^2} &+\sum_{1<|\Im(\rho_{E,2})|\leq T} \frac{x^{1-\frac c{\log(N_E (|T|+3))}}}{|\rho_{E,2}|^2}+  \sum_{\substack{|\Im(\rho_{E,2})| \leq 1 \\ \Re(\rho_{E,2}) > 0}} \frac {x^{1-\frac c{\log(4N_E)}}}{|\rho_{E,2}|} \\
&\ll x \frac{\log(TN_E)}{T}+ (x^{1-\frac c{\log(N_E(T+3))}}+\log N_E) \log N_E.
\end{align*}  
Selecting $T=\exp(\sqrt{\log x})$, we obtain the error term on the right hand side of \eqref{Sym2 unconditional}. 

As for the term $L'(0,\text{Sym}^2 f_E)/L(0,\text{Sym}^2 f_E)$, we combine the zero-free region with \cite[5.28]{IK}, and \eqref{equation bound logarithmic derivative} becomes
\begin{equation}
\frac{L'(1,\text{Sym}^2 f_E)}{L(1,\text{Sym}^2 f_E)} \ll \sum_{|\rho_{E,2}| <1} \frac 1{(\log q(\text{Sym}^2 f_E))^{-1}} +\log q(\text{Sym}^2 f_E) \ll (\log N_E)^2,
\end{equation}
from which we get using the functional equation that $L'(0,\text{Sym}^2 f_E)/L(0,\text{Sym}^2 f_E)\ll (\log N_E)^2$. 

The proof of \eqref{Sym2 unconditional} follows from combining these estimates with a similar  calculation to \eqref{equation playing with sym2}, in which we replace the application of the Riemann Hypothesis for $\zeta(s)$ with the Prime Number Theorem.

\end{proof}

\begin{proof}[Proof of Lemma \ref{lemma Sym3}]
The proof is similar to that of Lemma \ref{lemma Sym2}. By the work of Kim and Shahidi \cite{KiS2}, $L(\text{Sym}^3 E,s+\tfrac 32)=L(s,\text{Sym}^3 f_E)$ is the $L$-function of a cuspidal automorphic form on $GL(4)$. It follows that this function can be analytically continued to the whole complex plane, and the completed $L$-function \cite{CM}, given by 
$$ \Lambda(s,\text{Sym}^3 f_E):= 4 q(\text{Sym}^3 f_E)^{-\frac s2}(2\pi)^{-2s-2} \Gamma(s+\tfrac 12)\Gamma(s+\tfrac 32) L(s,\text{Sym}^3 f_E), $$
satisfies the functional equation $\Lambda(s,\text{Sym}^3 f_E)=\epsilon(\text{Sym}^3 f_E)\Lambda(1-s,\text{Sym}^3 f_E)$. Here, $q(\text{Sym}^3 f_E)=N_E^3$ and $\epsilon(\text{Sym}^3 f_E)=\pm 1$ \cite{CM}. The logarithmic derivative of $\Lambda(s,\text{Sym}^3 f_E)$ is given by 
$$ -\frac{L'(\text{Sym}^3 E,s+\tfrac 32)}{L(\text{Sym}^2 E,s+\tfrac 32)} = \sum_{k\geq 1} \sum_p \frac{(\alpha_p(E)^{3k} +\alpha_p(E)^{2k}\beta_p(E)^k +\alpha_p(E)^k\beta_p(E)^{2k} + \beta_p(E)^{3k}) \log p}{p^{ks}}.$$
This function is holomorphic at $s=1$ whether or not $E$ has complex multiplication (see Remark \ref{remark CM}), hence arguing as in \eqref{equation explicit formula Sym2} we obtain that the Riemann Hypothesis for $L(s,\text{Sym}^3 f_E)$ implies
$$ \sum_{\substack{ p  \\ k\geq 1}} (\alpha_p(E)^{3k} +\alpha_p(E)^{2k}\beta_p(E)^k +\alpha_p(E)^k\beta_p(E)^{2k} + \beta_p(E)^{3k})  h(p^{k}/x) \log p\ll x^{\frac 12} \log N_E. $$
The proof is achieved by trivially bounding the prime powers $k\geq 2$:
$$\sum_{\substack{ p \\ k\geq 2}} (\alpha_p(E)^{3k} +\alpha_p(E)^{2k}\beta_p(E)^k +\alpha_p(E)^k\beta_p(E)^{2k} + \beta_p(E)^{3k})  h(p^k/x) \log p\ll x^{\frac 12};$$
and by the following computation:
\begin{align*}
\sum_{\substack{ p}} &(\alpha_p(E)^{3} +\alpha_p(E)^{2}\beta_p(E) +\alpha_p(E)\beta_p(E)^{2} + \beta_p(E)^{3}) h(p^k/x) \log p \\&=\sum_{\substack{ p}} (\alpha_p(E)^{3}  + \beta_p(E)^{3}) h(p^k/x) \log p+\sum_{\substack{ p }} (\alpha_p(E) +\beta_p(E)) h(p^k/x) \log p+ O\bigg(\sum_{p\mid N_E} \log p\bigg)\\
&= \sum_{\substack{ p }} (\alpha_p(E)^{3}  + \beta_p(E)^{3}) h(p^k/x) \log p+ O(x^{\frac 12} \log N_E+\log N_E),
\end{align*}
by the Riemann Hypothesis for $L(s,f_E)$.

\end{proof}

\section{The average analytic rank in terms of a prime sum}

Consider for $d\neq 0$ squarefree the quadratic twists
$$ E_d : d y^2 = x^3+ax+b. $$

The curve $E_d$ has conductor dividing $d^2N_E$, and in the case where $(d,N_E)=1$ its conductor is exactly $d^2N_E$. Moreover, the $L$-function of $E_d$ is the Ranking-Selberg convolution of that of $E$ and of $L(s,\chi_d)$, that is if $L(E,s+\tfrac 12)=L(s,f_E)=\sum_{n\geq 1} \lambda_E(n) n^{-s}$, then\footnote{This follows from the following calculation: $$a_p(E_d)=p+1- \# (E_d)_p(\mathbb F_p) = \sum_{x\bmod p} \left( \frac dp \right) \left( \frac{x^3+ax+b}{p} \right) =\left( \frac dp \right) a_p(E).  $$  }
$$ L(E_d,s+\tfrac 12) = L(s,f_E\otimes \chi_d) = \sum_{n\geq 1} \frac{\lambda_E(n) \chi_d(n)}{n^s}.  $$
As is customary, we have denoted
$$ \chi_d(n) := \left( \frac dn \right). $$

We fix $g:(0,\infty) \rightarrow \mathbb R$ a function satisfying Property D (see Section \ref{section prime number theorems}). Note that for $k\geq 1$, the function $g_k(x):=g(x^k)$ also satisfies Property D, since
\begin{equation}
 \mathcal M g_k(s) = \int_{0}^{\infty} x^{s-1} g(x^k) dx = \frac 1k \int_0^{\infty} t^{\frac sk -1} g(t) dt = \tfrac 1k \mathcal M g(\tfrac sk), \label{equation Mellin of g_k}
\end{equation} 
which initially converges for $\Re(s)\geq 0$ and extends to a meromorphic function on $\mathbb C$ with possible simple poles at the points $s=0,-k,-2k,...$ The decay condition is trivial to check.

The central quantity we will study is the "prime sum"
\begin{equation}
\label{eq:def of S(D;P)}
S(D;P) := - \sumdw \sum_{p\leq P} \frac{\chi_d(p)a_p(E)\log p}{\sqrt p}   g(p/P).
\end{equation}

In this section we apply Lemmas \ref{lemma Sym2} and \ref{lemma Sym3} to show that under either of Hypotheses M or M$(\delta,\eta)$ for some $0<\delta<1$ and $0<\eta<\tfrac 12$, the quantity $S(D;P)$ is strongly linked with the average rank of $E_d$. Since these hypotheses are stated with a specific weight $g$, we will state the analogous bounds for a general function $g$ satisfying Property D:
\begin{align}
\sumdw  \sum_{\rho_d \notin \mathbb R} P^{\rho_d }\mathcal M g(\rho_d) &= o(D P^{\frac 12}); \label{equation M general} \\
\sumdw  \sum_{\rho_d \notin \mathbb R} P^{\rho_d }\mathcal M g(\rho_d) & =O_{E}(D^{1-\eta} P^{\frac 12}). \label{equation M eta general}
\end{align}

\begin{proposition}
\label{proposition average rank}
Assume that $L(E_d,s)$ has no real zeros in the critical strip, except possibly at the central point, and assume that \eqref{equation M general} holds in the range $  D^{2-\delta} \leq P \leq  2D^{2-\delta}$, for some $\delta>0$. Then in the same range we have
\begin{equation}
 P^{-\frac 12} S(D;P) =\mathcal Mg(\tfrac 12) \sumdw \left(r_{\text{an}}(E_d) - \frac 12 \right) +o_{D\rightarrow \infty}(D). 
 \label{equation averange rank unconditional}
\end{equation}

Assuming moreover that \eqref{equation M eta general} holds for some $0<\eta<\tfrac 12$, as well as the Riemann Hypothesis for $L(\text{Sym}^2E,s+1)$ and for $\zeta(s)$, we have in the same range the stronger estimate

\begin{equation}
 P^{-\frac 12} S(D;P) =\mathcal Mg(\tfrac 12) \sumdw \left(r_{\text{an}}(E_d) - \frac 12 \right) +O_E(P^{-\frac 16}D+D^{1-\eta}). 
 \label{equation averange rank GRH}
\end{equation}

Finally, assume that \eqref{equation M eta general} holds for some $0<\eta<\tfrac 12$ in the range $D^{2-\delta} \leq P \leq  2D^{2-\delta}$ for some $0<\delta<1$, and assume the Riemann Hypothesis for $\zeta(s)$ and for $L(\text{Sym}^kE,s+\tfrac k2)$ with $1\leq k\leq 3$. Then we have in the same range
\begin{equation}
P^{-\frac 12} S(D;P) =\mathcal Mg(\tfrac 12) \sumdw \left(r_{\text{an}}(E_d) - \frac 12 \right) +O_E(P^{-\frac 14}D+D^{1-\eta}). 
\label{equation averange rank GRH+}
\end{equation} 

\end{proposition}

\begin{remark}
The dependence on $E$ of the error terms in \eqref{equation averange rank unconditional}, \eqref{equation averange rank GRH} and \eqref{equation averange rank GRH+} can be determined explicitly from the following proof, provided the dependence on $E$ in \eqref{equation M general} and \eqref{equation M eta general} is known.

\end{remark}

\begin{remark}
The error term $O_E(P^{-\frac 14}D)$ in \eqref{equation averange rank GRH+} comes from the Riemann Hypothesis for $L(\text{Sym}^2E,s+\tfrac 12)$, and thus assuming the Riemann Hypothesis for $L(\text{Sym}^kE,s+\tfrac 12)$ with $k\geq 4$ will not yield a better error term.
\end{remark}

\begin{proof}

We first write the explicit formula for $L(E_d,s+\tfrac 12)$. Note that 
$$ -\frac{L(E_d,s+\tfrac 12)}{L(E_d,s+\tfrac 12)} = \sum_{\substack{ p \\ k\geq 1}} \frac{(\alpha_p(E_d)^k+\beta_p(E_d)^k)\log p}{p^{ks}},  $$
hence as in the proof of Lemma \ref{lemma Sym2} we obtain the formula
\begin{align}
\sum_{\substack{p \\ k\geq 1}}\chi_d(p)^k(\alpha_p(E)^k+\beta_p(E)^k) g(p^k/P) \log p   
&= -\sum_{\rho_d} P^{\rho_d} \mathcal M g(\rho_d) + O\left( \log (|d|N_E) \right). \label{equation explicit formula E}
\end{align}  
Working on the left hand side of this equation, we use the bounds $|\alpha_p(E)|,|\beta_p(E)| \leq 1$ to obtain
\begin{align}
\begin{split}
 \sum_{\substack{p \\ k\geq 2}}\chi_d(p)^k(\alpha_p(E)^k+&\beta_p(E)^k)g(p^k/P) \log p    =  \sum_{\substack{p \nmid d }}(\alpha_p(E)^2+\beta_p(E)^2)g(p^2/P) \log p    +O(P^{\frac 13}) \\
 &=   \sum_{\substack{p  }}(\alpha_p(E)^2+\beta_p(E)^2)g_2(p/P^{\frac 12}) \log p    +O(P^{\frac 13}+\log |d|)\\
 &= - \mathcal M g_2(1) P^{\frac 12} + O( P^{\frac 12} e^{-c \frac{ \frac 12\log P}{\sqrt{\frac 12\log P} + \log N_E}}\log N_E+ \log |d|),
 \end{split}
 \label{equation application of lemma Sym2}
\end{align} 
by Lemma \ref{lemma Sym2}. Note that $\mathcal M g_2(1) = \tfrac 12\mathcal Mg(\tfrac 12)$ (see \eqref{equation Mellin of g_k}), and hence it follows from \eqref{equation explicit formula E}, \eqref{equation application of lemma Sym2} and the identity $a_p(E)/\sqrt p = \alpha_p(E)+\beta_p(E)$ that
\begin{multline}
 \sum_{p \leq P} \frac{\chi_d(p) a_p(E) \log p}{\sqrt p} g(p/P) = 
\tfrac 12\mathcal M g(\tfrac 12) P^{\frac 12}-\sum_{\rho_d} P^{\rho_d}\mathcal Mg(\rho_d) \\+O( P^{\frac 12} e^{- \frac{c'\log P}{\sqrt{\log P} + \log N_E}}\log N_E+ \log (|d|N_E)).
 \label{equation avg rank}
\end{multline}
Note that the sum over zeros contains both real and nonreal zeros, however we are assuming that the only possible real zero $\rho_d$ of $L(E_d,s+\tfrac 12)$ is at the central point, that is $\rho_d=\tfrac 12$, which gives a contribution of $-r_{\text{an}}(E_d)P^{\frac 12}\mathcal Mg(\tfrac 12)$. We then obtain by summing \eqref{equation avg rank} over squarefree $0<|d|\leq D$ against the smooth weight $ w(d/D)$ the estimate
\begin{multline}
P^{-\frac 12}S(D;P)= \mathcal Mg(\tfrac 12)\sumdw \left(r_{\text{an}}(E_d)-\frac 12\right)+ \sumdw  \sum_{\rho_d\notin \mathbb R} P^{\rho_d-\frac 12}\mathcal M g(\rho_d)\\ +O\left(  De^{- \frac{c'\log P}{\sqrt{\log P} + \log N_E}}\log N_E+ \frac{D\log (DN_E)}{P^{\frac 12}}\right).
\label{equation application of Hypothesis M}
\end{multline}
The proof follows from applying \eqref{equation M general}.

To prove \eqref{equation averange rank GRH} and \eqref{equation averange rank GRH+}, we return to \eqref{equation application of lemma Sym2}, but perform a more precise calculation:
\begin{align}
 &\sum_{\substack{p \\ k\geq 2}}\chi_d(p)^k(\alpha_p(E)^k+\beta_p(E)^k)g(p^k/P)\log p \notag\\
  &=  \sum_{\substack{p   }}(\alpha_p(E)^2+\beta_p(E)^2)g_2(p/P^{\frac 12})\log p +\sum_{\substack{p  }}(\alpha_p(E_d)^3+\beta_p(E_d)^3)g_3(p/P^{\frac 13})\log p\notag \\&\hspace{10cm}+O(P^{\frac 14}+\log |d|).  \label{equation application of lemma Sym3}
\end{align} 
(Note that for $p\nmid d$, $\chi_d(p)^2=1$, and $\chi_d(p)^3\alpha_p(E)^3=\alpha_p(E_d)^3$.) Assuming the Riemann Hypothesis for $L(s,\text{Sym}^2f_E)$ and $\zeta(s)$, we obtain by an application of Lemma \ref{lemma Sym2} and by trivially bounding the second sum that \eqref{equation application of lemma Sym3} equals
$$ -\tfrac 12\mathcal Mg(\tfrac 12)P^{\frac 12} + O(P^{\frac 14}\log N_E) + O(P^{\frac 13})+O(P^{\frac 14}+\log |d|).  $$
Following the steps above and applying \eqref{equation M eta general} in \eqref{equation application of Hypothesis M} gives \eqref{equation averange rank GRH}. As for \eqref{equation averange rank GRH+}, we follow the same procedure, except that we apply Lemmas \ref{lemma Sym2} and \ref{lemma Sym3} to \eqref{equation application of lemma Sym3}, which gives that this quantity equals
$$ -\tfrac 12\mathcal M g(\tfrac 12)P^{\frac 12} + O(P^{\frac 14}\log N_E) + O(P^{\frac 16}\log (|d|N_E))+O(P^{\frac 14}+\log |d|).  $$
\end{proof}

\section{Quadratic twists: An upper bound for the prime sum $S(D;P)$}

In this section we follow the arguments of Katz and Sarnak \cite{KS}, which were inspired from Iwaniec's work \cite{I1}. The goal is to give an upper bound on the prime sum $S(D;P)$ (see \eqref{eq:def of S(D;P)}), which by Proposition \ref{proposition average rank} will yield information about the average analytic rank of $E_d$. Let us first give a consequence of ECRH. We fix a function $g:(0,\infty)\rightarrow \mathbb R$ satisfying Property D (see Section \ref{section prime number theorems}). 

\begin{lemma}
\label{lemma Riemann bound}
Assume ECRH. We have for $m\neq 0$ and $P,t\geq 1$ the estimates
\begin{equation}
\sum_{p\leq t} \frac{ a_p(E) \log p}{\sqrt p} \left( \frac m p\right)  \ll t^{\frac 12} (\log t)\log (2t|m|N_E); 
\label{equation s_y without weight}
\end{equation} 
\begin{equation}
 \sum_{p} \frac{ a_p(E) \log p}{\sqrt p} \left( \frac m p\right) g(p/P) \ll P^{\frac 12} \log (2|m|N_E). 
 \label{equation s_y with weight}
\end{equation} \end{lemma}
\begin{proof}
First note that $a_p(E)\big( \tfrac m p\big)=a_p(E)$, where $E_m$ is the quadratic twist $m y^2=x^3+ax+b$. Note also that $E_m$ is isomorphic to $E_{m'}$, where $m'$ is the squarefree part of $m$. Applying the Riemann Hypothesis to $L(E_{m'},s)$, which is modular, we obtain by \cite[Theorem 5.15]{IK} that
$$\sum_{p\leq t} \frac{ a_p(E_y) \log p}{\sqrt p} \ll t^{\frac 12}(\log t)\log (2t|m|N_E). $$
The proof follows by bounding trivially the contribution of the primes dividing $mN_E$. As for \eqref{equation s_y with weight}, it follows along the similar lines, except that we use the following explicit formula:
$$ \sum_{\substack{ p \\ k\geq 1}} (\alpha_p(E)^k+\beta_p(E)^k)g(p^k/P)\log p = - \sum_{\rho_d} P^{\rho_d}\mathcal M g(\rho_d) + O(\log (|d|N_E)).$$
\end{proof}

We will need the following lemma of \cite{FPS}.
\begin{lemma}
\label{lemma count of squarefree}
Fix $n\geq 1$ and $\epsilon>0$. Under the Riemann Hypothesis we have the estimate

\begin{multline*}
\sumdw \left(\frac dn\right)=  \kappa(n) \frac{D}{\zeta(2)} \int_{\mathbb R} w(t) dt  \prod_{p\mid N_E} \bigg(1+ \frac{\left(\frac pn \right)}{p} \bigg)^{-1} \prod_{p\mid n}  \left(1+\frac 1{p} \right)^{-1}\\+O_{\epsilon}((N_E)^{\epsilon}|n|^{\frac 38(1-\kappa(n))+\epsilon}D^{\frac 14+\epsilon}),
\end{multline*}
where
$$ \kappa(n) := \begin{cases}
1 &\text{ if } n = \square, \\
0 &\text{ otherwise.}
\end{cases}$$

\end{lemma}

We are now ready to bound $S(D;P)$.

\begin{proposition}
\label{proposition KS}
Assume ECRH and fix $0<\delta < 1$. Then in the range $ D^{\delta}\leq P\leq  D^{2-\delta}$ we have
$$ P^{-\frac 12}S(D;P) \ll_{\epsilon} N_E^{\epsilon}P^{\frac 14+\epsilon}D^{\frac 12}=
o_{D\rightarrow \infty}(D). $$
\end{proposition}
\begin{proof}

The proof is an adaptation of that of \cite[Theorem (B)]{KS}. 
We first turn the sum over squarefree $d$ into a sum over all integers using the identities $\mu^2(d) = \sum_{a^2\mid d} \mu(a)$; $I_{(d,N_E)=1} = \sum_{c\mid (d,N_E)}\mu(c)$. We will use Lemma \ref{lemma count of squarefree} to bound the contribution of $p=2$ in \eqref{equation you should bound 2 here}. Denoting by $[a,c]$ the least common multiple of $a$ and $c$, we compute
\begin{align} S(D;P) &= - \sumdw \sum_{p} \frac{\chi_d(p)a_p(E)\log p}{\sqrt p}   g(p/P)\notag \\&
=- \sum_{p}  \frac{a_p(E)\log p}{\sqrt p}   g(p/P)\sum_{ d \neq 0} \sum_{a^2 \mid d} \mu(a)\sum_{c \mid (d,N_E)} \mu(c) w\left( \frac dD \right)\left( \frac dp \right) \label{equation you should bound 2 here}
\\
&= - \sum_{a\geq 1} \mu(a) \sum_{c \mid N_E} \mu(c) \sum_{\substack{p\nmid 2ac}}  \frac{a_p(E)\log p}{\sqrt p}   g(p/P)\sum_{ b \in \mathbb Z}  w\left( \frac {[a^2,c]b}D \right)\left( \frac {[a^2,c]b}p \right)+O_{\epsilon}(D^{\frac 12}N_E^{\epsilon}),\notag \end{align}
since the primes dividing $ac$ give a zero contribution, and the prime $p=2$ contributes 
$$ \ll \sum_{a\leq D^{\frac 12}} \tau(N_E)+ \sum_{a > D^{\frac 12} } \tau(N_E) \frac{D}{a^2} \ll \tau(N_E) D^{\frac 12}. $$
We then have that \eqref{equation you should bound 2 here} equals
$$
 T_1(A)+T_2(A)+O_{\epsilon}(D^{\frac 12+\epsilon}N_E^{\epsilon}).$$
where, fixing a parameter $A\geq 1$,
$$ T_2(A):=-\sum_{a > A} \mu(a) \sum_{c \mid N_E} \mu(c)\sum_{\substack{ p\nmid 2ac}}  \frac{a_p(E)\log p}{\sqrt p}  g(p/P) \sum_{ b \in \mathbb Z}  w\left( \frac {[a^2,c]b}D \right)\left( \frac {[a^2,c]b}p \right),$$ and $T_1(A)$ is the same sum with $a\leq A$. 
Applying Lemma \ref{lemma Riemann bound} gives the bound
\begin{align*}
|T_2(A)|&\leq \sum_{a > A}\sum_{c \mid N_E} \sum_{ b \in \mathbb Z} w\left( \frac {[a^2,c]b}D \right) \bigg| \sum_{\substack{ p\neq 2 }} \left( \frac {[a^2,c]b}p \right) \frac{a_p(E)\log p}{\sqrt p}   g(p/P) \bigg| \\
&\ll   \sum_{a > A} \sum_{c \mid N_E}\sum_{ b \in \mathbb Z} w\left( \frac {[a^2,c]b}D \right) ( P^{\frac 12}\log(abc N_E) + 1) \\ &\ll  P^{\frac 12} \tau(N_E)\sum_{a > A} \frac {D\log (ADN_E)}{a^2} \ll DP^{\frac 12} \frac{ \tau(N_E)\log (ADN_E)}A.
\end{align*}   
To treat  $T_1(A)$, we will use Gauss sums. Let $\epsilon_n=\tfrac{1+i}2 \chi_0(n)+\tfrac{1-i}2 \chi_1(n)$, where $\chi_0$ and $\chi_1$ are respectively the trivial and the nontrivial character modulo $4$. Hence $\epsilon_p=1$ if $p\equiv 1 \bmod 4$ and $-i$ if $p\equiv 3 \bmod 4$.
Then, for any $a\geq 1$, $c\mid N_E$ and $p\nmid 2ac$ we have
\begin{align*} \sum_{b\in \mathbb Z }  w\left( \frac {[a^2,c]b}D \right)  \left( \frac{[a^2,c]b}{p}\right)&= \sum_{b\in \mathbb Z}  w\left( \frac {[a^2,c]b}D \right) \frac{\overline{\epsilon_p}}{p^{\frac 12}} \sum_{x\bmod p} \left( \frac xp\right) e\left(\frac{[a^2,c]b x}p \right) \\&= \frac{\overline{\epsilon_p}}{p^{\frac 12}}  \sum_{x\bmod p} \left( \frac xp\right) \sum_{b\in \mathbb Z}  w\left( \frac {[a^2,c]b}D \right)  e\left(\frac{[a^2,c]b x}p \right).
\end{align*}
We now transform the sum over $b$ using Poisson summation:
$$  \sum_{b\in \mathbb Z}  w\left( \frac {[a^2,c]b}D \right)  e\left(\frac{[a^2,c]b x}p \right) =\frac D{[a^2,c]} \sum_{m \in \mathbb Z} \hat w\left( D\left( \frac m{[a^2,c]} - \frac xp\right)\right). $$

Inserting this into the definition of $T_1(A)$ gives the identity
$$T_1(A) = -D \sum_{a\leq A} \sum_{c\mid N_E} \frac {\mu(a)\mu(c)}{[a^2,c]}\sum_{\substack{p\nmid 2ac}} \frac{\overline{\epsilon_p} a_p(E)\log p}{p}   g(p/P)\sum_{x\bmod p} \left( \frac xp\right) \sum_{m \in \mathbb Z} \hat w\left( D\left( \frac m{[a^2,c]} - \frac xp\right)\right).$$

Note that as $x$ runs through $\mathbb Z/p\mathbb Z$ and $m$ runs through $\mathbb Z$, $y:=pm-[a^2,c]x$ runs throught $\mathbb Z$. Indeed, $(p,ac)=1$ implies that the map $(x,m)\mapsto pm-[a^2,c] x$ is a bijection between $\mathbb Z/p\mathbb Z \times \mathbb Z$ and $\mathbb Z$. We deduce that
\begin{align} T_1(A) &= -D \sum_{a\leq A} \sum_{c\mid N_E} \frac {\mu(a)\mu(c)}{[a^2,c]} \sum_{\substack{ p\nmid 2ac}} \frac{\overline{\epsilon_p} a_p(E)\log p}{p}   g(p/P)\sum_{y \in \mathbb Z} \left( \frac {-[a^2,c]y}p\right)  \hat w\left( \frac {Dy}{[a^2,c]p}\right)\\
&= -D \sum_{a\leq A} \sum_{c\mid N_E} \frac {\mu(a)\mu(c)}{[a^2,c]} \sum_{y \in \mathbb Z}   \int_1^{P^{1+\epsilon/2}}  \hat w\left( \frac {Dy}{[a^2,c]t}\right) dS_y(t)+ O_{N}(\tau(N_E)DP^{-N}),
\label{equation integral for T_1}
\end{align}
where 
$$ S_y(t) := \sum_{p\leq t} \left( \frac {-4y[a^2,c]}p\right) \frac{\overline{\epsilon_p} a_p(E)\log p}{p}   g(p/P), $$
and the error term comes from the decay of $g(x)$ in Property D.  By the definition of $\epsilon_n$ and by applying summation by parts in Lemma \ref{lemma Riemann bound}, Hypothesis ECRH implies
 $$S_y(t) \ll (\log t) \log (|acy|tN_E+2).$$ 
 Moreover, since $w$ is Schwartz we have the bounds
$$ \hat w(\xi), \hat w'(\xi) \ll_{\epsilon} \min(\xi^{-1-\epsilon},\xi^{-2-\epsilon}). $$
Applying these bounds after a summation by parts, we obtain that the first term in \eqref{equation integral for T_1} is 
\begin{align*}
&=  -D\sum_{a\leq A} \sum_{c\mid N_E} \frac {\mu(a)\mu(c)}{[a^2,c]}  \sum_{y\in \mathbb Z}\hat w\left( \frac {Dy}{[a^2,c]t}\right)  S_y(t) \Bigg|_1^{P^{1+\epsilon/2}} \\ &\hspace{2cm}-D  \sum_{a\leq A} \sum_{c\mid N_E} \frac {\mu(a)\mu(c)}{[a^2,c]} \int_1^{P^{1+\epsilon/2}}   \sum_{y\in \mathbb Z} \frac{DyS_y(t)}{[a^2,c]}   \hat w'\left( \frac {Dy}{a^2t}\right) \frac{dt}{t^2} \\
&\ll_{\epsilon}   D\sum_{a\leq A} \frac 1{a^2} \sum_{c\mid N_E} \sum_{y\in \mathbb Z}  (\log(acP|y|N_E+2))(\log P) \left( \frac {Dy}{[a^2,c]P^{1+\epsilon/2}}\right)^{-1-\epsilon/3} 
\\&\hspace{1cm}+D^2\sum_{a\leq A} \frac 1{a^4} \sum_{c\mid N_E}  \int_1^{P^{1+\epsilon/2}}   \sum_{y\in \mathbb Z} y (\log(acP|y|N_E+2))(\log P) \left( \frac {Dy}{[a^2,c]t}\right)^{-2-\epsilon/2} \frac{dt}{t^2} \\
& \ll_{\epsilon} (AN_E)^{1+\epsilon}P^{1+\epsilon} \log N_E.
\end{align*}
Combining our estimates for $T_1(A)$ and $T_2(A)$ we obtain that
\begin{multline*} \frac{S(D;P)}{DP^{\frac 12}} = D^{-1}P^{-\frac 12}(T_1(A) +T_2(A)+O_{\epsilon}(D^{\frac 12}N_E^{\epsilon})) \\ \ll_{\epsilon} \frac {\tau(N_E)(\log ADN_E)^2}A+ \frac{(AN_E)^{1+\epsilon}P^{\frac 12+\epsilon}}D\log N_E+ \frac {N_E^{\epsilon}}{P^{\frac 12}D^{\frac 12}}. 
\end{multline*}
(The error term in \eqref{equation integral for T_1} is absorbed by the error term $N_E^{\epsilon}/P^{\frac 12}D^{\frac 12-\epsilon}$). The claimed estimate follows from taking $A:=D^{\frac 12}P^{-\frac 14}\geq 1$.
\end{proof}

\section{Proof of the main results}

We first give an effective equidistribution result for the root number $\epsilon(E_d)$.

\begin{lemma}
\label{lemma equidistribution root number coprime}
Fix $w$ a Schwartz function. We have the bounds

$$ \sumdw \epsilon(E_d) \ll_{\epsilon}D^{\frac 12+\epsilon}N_E^{\frac 14+\epsilon}, \hspace{1cm}  \underset{\substack{ 0< |d| \leq D \\ (d,N_E)=1}}{\sum \nolimits^*}\epsilon(E_d) \ll D^{\frac 12}N_E^{\frac 14}(\log N_E)^{\frac 12}.  $$

\end{lemma}
\begin{remark}
One can improve the dependence on $N_E$ in these estimates by using Burgess's bound (see \cite{GV}), however this is not important for our purposes since $E$ is fixed.
\end{remark}
\begin{proof}
We prove the second estimate; the first follows along similar lines. Note that by \cite[(23.48)]{IK}, for $(d,N_E)=1$ we have $\epsilon(E_d)= (\tfrac d{-N_E})\epsilon_E$. Now,
\begin{align*}
 \underset{\substack{ 0< |d| \leq D \\ (d,N_E)=1}}{\sum \nolimits^*} \left(\frac d{-N_E} \right) = \sum_{\substack{a\leq D^{\frac 12} \\ (a,N_E)=1}} \mu(a) \sum_{\substack{ 0< |\ell| \leq D/a^2 }}  \left( \frac {\ell}{-N_E} \right).
\end{align*}
We then split the sum into a sum over $a\leq A$ and $A<a\leq D^{\frac 12}$, with $A= D^{\frac 12} N_E^{-\frac 14}(\log N_E)^{-\frac 12}$. The first of these sums is bounded using the Poly\`a-Vinogradov Inequality, and the second using the trivial bound.
\end{proof}

\begin{proof}[Proof of Theorem \ref{theorem main twists} and Corollary \ref{first corollary}.]

Let $\delta>0$ be given by Hypothesis M, and for $D\geq 1$, set $P=P(D)=D^{2-\delta}$. Select $g(x)=\max(1-x,0),$ which satisfies Property D. On one hand, Proposition \ref{proposition average rank} gives the estimate
\begin{equation}
 P^{-\frac 12} S(D;P) =\mathcal Mg(\tfrac 12) \sumdw \left(r_{\text{an}}(E_d) - \frac 12 \right) +o_{D\rightarrow \infty}(D). 
\end{equation}
On the other hand, Proposition \ref{proposition KS} shows that
$$ P^{-\frac 12} S(D;P) = o_{D\rightarrow \infty}(D). $$
It follows that
\begin{equation}
\sumdw \left(r_{\text{an}}(E_d) - \frac 12 \right) =o_{D\rightarrow \infty}(D).
\label{equation average rank 1/2}
\end{equation}  
We will use the following notation, for $k\geq 0$:
$$S_k(D) := \underset{\substack{(d,N_E)=1 \\ r_{\text{an}}(E_d)=k} }{\sum\nolimits^*} \wt;  \hspace{1cm} S^{\text{odd}}_{\geq k}(D) :=  \underset{\substack{(d,N_E)=1 \\ r_{\text{an}}(E_d)\geq k \\ r_{\text{an}}(E_d)\equiv 1 \bmod 2 } }{\sum\nolimits^*} \wt,$$
and similarly for $S^{\text{even}}_{\geq k}(D)$. Now, $r_{\text{an}}(E_d)$ is even if $\epsilon(E_d)=1$, and odd if $\epsilon(E_d)=-1$. We now apply Lemma \ref{lemma equidistribution root number coprime}. Summing $(1+\epsilon(E_d))$ and $(1-\epsilon(E_d))$ against the weight $w(d/D)$ we obtain that
\begin{equation}
S^{\text{even}}_{\geq 0}(D) \sim \frac{W(D)}2; \hspace{1cm} S^{\text{odd}}_{\geq 1}(D) \sim \frac{W(D)}2,
\label{equation odd and even}
\end{equation}  
where 
$$ W(D):= \sumdw \sim D \int_{\mathbb R} w(t)dt. $$
We then have by \eqref{equation average rank 1/2} that
\begin{align*}
\frac{W(D)}2 \sim \sumdw r(E_d) &\geq 0 S_0(D) + S_1(D) + 2S^{\text{even}}_{\geq 2}(D) + 3S^{\text{odd}}_{\geq 3}(D) \\
&= 2S^{\text{even}}_{\geq 2}(D) + S^{\text{odd}}_{\geq 1}(D) + 2S^{\text{odd}}_{\geq 3}(D),
\end{align*}   
hence 
\begin{equation}
2(S^{\text{even}}_{\geq 2}(D) + S^{\text{odd}}_{\geq 3}(D)) \leq \frac{W(D)}2 - S^{\text{odd}}_{\geq 1}(D) +o(W(D)) = o(D). 
\label{equation no curves of high rank}
\end{equation} 
We now show that this implies that the proportion of curves of rank $\geq 2$ is zero. Since $w(0)\neq 0$ and $w$ is smooth, there exists $\eta>0$ and $\theta>0$ such that $w(x) \geq \theta$ for all $|x|\leq \eta$. We then have by nonnegativity of $w$ that
\begin{equation}
\underset{\substack{ 0<|d| \leq D\\(d,N_E)=1 \\ r_{\text{an}}(E_d)\geq 2} }{\sum\nolimits^*} 1 \leq \theta^{-1} \underset{\substack{ (d,N_E)=1 \\ r_{\text{an}}(E_d)\geq 2} }{\sum\nolimits^*} w \left(  \frac d{\eta^{-1} D}\right)  = \theta^{-1}o(\eta^{-1}D) = o(D), 
\label{equation no curves high rank without weight}
\end{equation}
by \eqref{equation no curves of high rank}; hence the number of curves of rank $\geq 2$ is negligible. Applying Lemma \ref{lemma equidistribution root number coprime} again, one obtains
\begin{equation}
\underset{\substack{ 0<|d| \leq D\\(d,N_E)=1 \\ r_{\text{an}}(E_d)=0} }{\sum\nolimits^*}1 \sim \frac{N(D)}2; \hspace{1cm}  \underset{\substack{ 0<|d| \leq D\\(d,N_E)=1 \\ r_{\text{an}}(E_d)=1} }{\sum\nolimits^*}1 \sim \frac{N(D)}2.
\label{equation rank 0 and 1 no weight}
\end{equation}  
That is, $50\%$ of the curves have rank $0$ and the remaining $50\%$ have rank $1$. 

To compute the average rank without the weight $w(d/D)$ we first see that
\begin{align*}
 \underset{\substack{ (d,N_E)=1 \\ r_{\text{an}}(E_d)\geq 2} }{\sum\nolimits^*} w\left( \frac dD \right) r(E_d) &=\underset{\substack{ (d,N_E)=1  } }{\sum\nolimits^*} w\left( \frac dD \right) \left( r(E_d) - \frac 12 \right) +\frac{W(D)}2 - \underset{\substack{ (d,N_E)=1 \\ r_{\text{an}}(E_d)=1} }{\sum\nolimits^*} w\left( \frac dD \right) \\&=o(D),
\end{align*}
 by \eqref{equation average rank 1/2}, \eqref{equation odd and even} and \eqref{equation no curves of high rank}. We then combine this with the fact that $w(x)\geq \theta$ for $|x|\leq \eta$:
 \begin{equation*}
\underset{\substack{ 0<|d| \leq D\\(d,N_E)=1 \\ r_{\text{an}}(E_d)\geq 2} }{\sum\nolimits^*} r(E_d) \leq \theta^{-1} \underset{\substack{ (d,N_E)=1 \\ r_{\text{an}}(E_d)\geq 2} }{\sum\nolimits^*} w \left(  \frac d{\eta^{-1} D}\right)r(E_d)  = \theta^{-1}o(\eta^{-1}D) = o(D). 
\end{equation*}
Combining this with \eqref{equation rank 0 and 1 no weight} implies that the average rank is exactly $1/2$.

Finally, the results of Gross-Zagier \cite{GZ}, Kolyvagin \cite{Ko} and others imply that the Birch and Swinnerton-Dyer Conjecture holds for the curve $E_d$ whenever $r_{\text{an}}(E_d)\leq 1$, and as we have shown in \eqref{equation rank 0 and 1 no weight}, this holds for almost all elliptic curves $E_d$.

\end{proof}

Our results also apply to the family $\{ E_d: \mu^2(d)=1 \}$. If $N_E$ is squarefree, then we can show that the root number $\epsilon(E_d)$ is equidistributed in this family.
\begin{lemma}
\label{lemma equidistribution root number}
Assume that $N_E$ is squarefree. Then we have 
$$ \underset{\substack{0< |d| \leq D }}{\sum \nolimits^{*} }  \epsilon(E_d)  \ll_E D^{\frac 12}.$$
\end{lemma}

\begin{proof}
It follows from \cite[(23.48)]{IK} that if $N_E$ is squarefree then we have the equality
$$ \epsilon(E_d)=\chi_d(-N_E/(d,N_E)) \mu((d,N_E)) \lambda_{E}((d,N_E)) \epsilon(E). $$
Here, $L(s,f_E) = \sum_n \lambda_E(n) n^{-s}$. We therefore have
\begin{align*}
 \underset{\substack{0< |d| \leq D }}{\sum \nolimits^{*} }  \epsilon(E_d)&= \epsilon(E) \sum_{ a\mid N_E} \mu(a) \lambda_{E}(a)  \underset{\substack{0< |d| \leq D \\ (d,N_E)=a }}{\sum \nolimits^{*} }  \chi_d(-N_E/a)\\
 &= \epsilon(E) \sum_{ a\mid N_E} \mu(a) \lambda_{E}(a)\underset{\substack{ 0< |k| \leq D/a \\ (k,N_E)=1  }}{\sum \nolimits^{*} }  \left( \frac{ka}{-N_E/a} \right) \\
 &= \epsilon(E)  \sum_{ a\mid N_E} \mu(a)\left( \frac{a}{-N_E/a} \right) \lambda_{E}(a) \sum_{\ell \mid N_E} \mu(\ell) \underset{\substack{ 0< |m| \leq D/(a\ell)  }}{\sum \nolimits^{*} } \left( \frac{m \ell}{-N_E/a} \right) \\
 & \ll D^{\frac 12}N_E^{\frac 14} \tau(N_E)^2 (\log N_E)^{\frac 12},
\end{align*}
by Lemma \ref{lemma equidistribution root number coprime}.
\end{proof}

\begin{proof}[Proof of Theorem \ref{theorem main error term}]
Letting $P=D^{2-\delta}$ with $\delta=\delta_{\eta}>0$ coming from Hypothesis M$(\eta)$, Proposition \ref{proposition average rank} gives the following estimate:
\begin{equation*}
\mathcal M g(\tfrac 12) \sumdw \left(r_{\text{an}}(E_d) - \frac 12 \right) =P^{-\frac 12} S(D;P)+O_E(D^{\frac 12+\frac{\delta}4}+D^{1-\eta}).
\end{equation*} 
With this choice of $P$, Proposition \ref{proposition KS} yields the bound
\begin{equation*}
\frac 4{3} \sumdw \left(r_{\text{an}}(E_d) - \frac 12 \right) \ll_{E} D^{1-\frac{\delta}4+o(1)}+D^{\frac 12+\frac{\delta}4}+D^{1-\eta}\ll D^{1-\frac{\delta}4+o(1)}+D^{1-\eta},
\end{equation*} 
since $\delta<1$. The remaining of the proof is similar to that of Theorem \ref{theorem main twists}.

\end{proof}

\begin{proof}[Proof of Corollary \ref{corollary curves of high rank}]
The proof is very similar to that of Corollary \ref{first corollary}. \end{proof}

\begin{proof}[Proof of Theorem \ref{theorem all curves}]
The idea is very similar to the proof of Theorem \ref{theorem main twists}. We fix a function $g$ with Property D and study the following prime sum with two different techniques:
\begin{equation}
\label{eq:def of S(A,B;P)}
S(A,B;P) := - \sumab \sum_{p} \frac{a_p(E_{a,b})\log p}{\sqrt p}   g(p/P).
\end{equation}

We first add all prime powers to $S(A,B;P)$ in order to apply the Explicit Formula. The squares of primes are treated using $L(\text{Sym}^2E_{a,b},s)$, as in Lemma \ref{lemma Sym2}:
\begin{multline*}\sum_{\substack{p \\ k\geq 1 }} (\alpha_p(E_{a,b})^{2k}+\alpha_p(E_{a,b})^k\beta_p(E_{a,b})^k+\beta_p(E_{a,b})^{2k}) g(p^{2k}/x^{2k})\log p \\ \ll xe^{-c\log x/(\sqrt{\log x} + \log (N_{E_{a,b}}))} \log (N_{E_{a,b}}),
\end{multline*}
from which we obtain using similar arguments that
\begin{align*}
\sum_{p } (\alpha_p(E_{a,b})^2+\beta_p(E_{a,b})^2)   g(p^2/P)\log p&= -\sum_{p }   g_2(p/P^{\frac 12})\log p +o(P^{\frac 12}) \\
&= -\tfrac 12\mathcal M g(\tfrac 12) P^{\frac 12}+o(P^{\frac 12}),
\end{align*}
where as before $g_k(t)=g(t^k).$
Trivially bounding the higher prime powers, we obtain
\begin{align*}
-\sum_{p} \frac{a_p(E_{a,b})\log p}{\sqrt p}   g(p/P)&=-\tfrac 12\mathcal Mg(\tfrac 12)P^{\frac 12}-  \sum_{\substack{p\\ k\geq 1}} (\alpha_p(E_{a,b})^k+\beta_p(E_{a,b})^k)  g(p^k/P)\log p+o(P^{\frac 12}) \\
&= -\tfrac 12\mathcal Mg(\tfrac 12)P^{\frac 12} + \sum_{\rho_{a,b}} \mathcal M g(\rho_{a,b}) P^{\rho_{a,b}} + o(P^{\frac 12})\\
&= \mathcal Mg(\tfrac 12)P^{\frac 12}(r_{\text{an}}(E_{a,b})-\tfrac 12)+ \sum_{\rho_{a,b}\notin \mathbb R} \mathcal M g(\rho_{a,b}) P^{\rho_{a,b}} + o(P^{\frac 12}).
\end{align*} 
We then obtain by summing over $a$ and $b$ against the weight $w$ that Hypothesis \ref{hypothesis all curves} implies the estimate 
\begin{equation}
P^{-\frac 12}S(A,B;P)=\mathcal Mg(\tfrac 12) \sumab (r_{\text{an}}(E_{a,b})-\tfrac 12) +o(AB). 
\label{equation average rank all curves}
\end{equation}

We now show that $S(A,B;P)$ is negligible compared to $AB$, from which it will follow that the average analytic rank is $\tfrac12$, by arguments analogous to the proof of Theorem \ref{theorem main twists}. This is a direct adaptation of \cite[Lemma 5.2]{Yo}. The two differences between the prime sum $P(E;\phi)$ and the inner sum in \eqref{eq:def of S(A,B;P)} is that the test function in the current paper is of the form $g(p/P)$, whereas the test function in \cite{Yo} is of the form $\phi(\log p/\log P)$, and the support of $e^{-t}$ is not compact. For these reasons, one cannot apply directly \cite[Lemma 5.2]{Yo}.
Instead we follow the proof, and see that the end result is the bound 
$ S(A,B;P) \ll (AB)^{1-\eta},$ for some $\eta>0$ which depends on the value of $\delta$ for which Hypothesis \ref{hypothesis all curves} . 

In our situation, the sum over primes is an infinite sum. We could have chosen a smooth test function having compact support, however we preferred to be specific since anyway the exponentially-decaying weight makes the terms with $p\geq P^{1+\epsilon}$ negligible:
$$\sum_{(a,b) \neq (0,0) } w\left( \frac aA,\frac bB \right) \sum_{p\geq P^{1+\epsilon}} \frac{a_p(E_{a,b})\log p}{\sqrt p}   e^{-p/P} \ll AB e^{-P^{\epsilon}}\ll_N X^{-N}.$$
Here the sum is over $a,b \geq 0$, not both zero. We will first bound this sum, and then argue as in \cite[Section 5.6]{Yo} to bound $S(A,B;P)$.

The analogue of \cite[(14)]{Yo} clearly follows from \cite[(15)]{Yo}.
In \cite[(20)]{Yo}, we need to replace $F(u,v)$ by
$$  F(u,v)= \left( \frac v P\right)^{-\frac 32} g\left(\frac{ud_0}H,\frac kK,\frac vV\right) \widehat w\left( \frac{ud_0 A}v,\frac{kB}v \right) e^{-v/P}.$$

The conditions of \cite[Lemma 5.7]{Yo} are satisfied in a very analogous way. Indeed taking partial derivatives of $e^{-v/P}$ in $v$ multiplies this quantity by $P^{-1}$. Moreover, $F(u,v)$ and its partial derivatives are exponentially small in $X$ for $v> P^{1+\epsilon}$, and are bounded in a similar way as in \cite{Yo} in the range $v\leq P^{1+\epsilon}$. Also, \cite[Lemma 5.7]{Yo} does not have a restriction on the support of $F(u,v)$, and therefore applies to our situation, that is we have the analogue of \cite[Corollary 5.3]{Yo}.

Having bounded the sum over primes, which was the only difference between our situation and that of \cite{Yo}, we now look at Cases 1-4 in \cite{Yo}. Cases 1-3 depend only on \cite[(15)]{Yo} and \cite[Corollary 5.3]{Yo}, and thus the proof remains identical. In Case 4 we can clearly restrict to dyadic intervals $Q < p \leq 2Q$ with $Q\leq P^{1+\epsilon}$. \cite[Lemma 5.4]{Yo} then applies, from which the analogue of \cite[Corollary 5.5]{Yo} follows. The sum over primes is now treated, and thus the rest of the proof is identical.

We conclude that 
$$\sum_{(a,b) \neq (0,0) } w\left( \frac aA,\frac bB \right) \sum_{p} \frac{a_p(E_{a,b})\log p}{\sqrt p}   e^{-p/P} \ll (AB)^{1-\eta},$$
for some $\eta>0$. The arguments of \cite[Section 5.6]{Yo} then show that the following bound holds by similar arguments:
$$ S(A,B;P) = o(AB). $$
Combining this with \eqref{equation average rank all curves} and using the estimate $W(A,B)\asymp AB$, we deduce that the average analytic rank of the elliptic curves $E_{a,b}$ is exactly $\tfrac 12$.
\end{proof}

\section{An upper bound on the algebraic rank of elliptic curves in  families of quadratic twists}

In this section we give an upper bound on the algebraic rank of the elliptic curve (over $\mathbb Q$)
$$ E_d: dy^2=x^3+ax+b $$
which depends only on $E$. This was pointed out to me by Silverman. The bound in question is $r_{\text{al}}(E_d)\leq 18\omega(d)+O_E(1)$, and consequently $r_{\text{al}}(E_d)\ll_E \log \log d$ for almost all $d$. Note that for most $d$, this is sharper than Mestre's conditional bound on the analytic rank $r_{\text{an}}(E_d)\ll_E \log d/\log\log d$. 

\begin{lemma}
Fix an elliptic curve $E$ over $\mathbb Q$, and consider the quadratic twists $E_d$. Then one has the bound
$$ r_{\text{al}}(E_d) \leq 18\omega(d)+O_E(1). $$
\label{lemma bound algebraic rank}
\end{lemma}

\begin{proof}
Denote by $E[2]$ the $2$-torsion of $E$ (over $\mathbb C$), and let $K=\mathbb Q(E[2])$, which is a finite Galois extension of $\mathbb Q$ (since $E$ is smooth). Considered as an elliptic curve over $K$, $E$ satisfies the condition $E[2]\subset E(K)$. This condition is then automatically satisfied for all quadratic twists $E_d$ (that is $E[2]\subset E_d(K))$, since a direct calculation shows that for $d \neq 0$, $\mathbb Q(E_d[2])=\mathbb Q(E[2])$. 

We will bound the rank of $E_d(K)$, which will be sufficient for our purposes since it gives an upper bound on the rank of $E_d(\mathbb Q)$. Moreover, since $E_d(K)\simeq \mathbb Z^{r_{\text{al}}(E_d(K))} \oplus E_d(K)_{\text{tors}}$, it follows that $2^{r_{\text{al}}(E_d(K))}
$ is bounded above by the cardinality of $E_d(K)/2E_d(K)$. As for this last quantity, \cite[Exercise 8.1]{Sil}) gives the bound
$$ \# E_d(K)/2E_d(K) \leq 2^{ 2 \# \{ \mathfrak p \subset \mathcal O_K: E_d \text{ has bad reduction at } \mathfrak p\}+c_K} \leq 2^{2 \omega(d^2N_E)[K:\mathbb Q]+c_K},  $$
since the conductor of $E_d$ divides $d^2N_E$. Notice that the constant $c_K$ depends on the class number of $K$, and thus the fact that the base field $K$ is independent of $d$ is crucial in this proof. The splitting field of a cubic polynomial over $\mathbb Q$ has degree at most $9$, and hence the result follows since $K$ depends only on $E$.

\end{proof}

We now deduce the result on the average algebraic rank.

\begin{proof}[Proof of Theorem \ref{theorem average algebraic rank}]

In a similar way to the proof of Theorem \ref{theorem main twists} and Theorem \ref{theorem main error term}, our hypotheses imply that 
\begin{equation*}
\sumdw \left(r_{\text{an}}(E_d) - \frac 12 \right) =o\left( \frac D {\log \log D}\right).
\end{equation*} 
Note that no hypothesis on symmetric power $L$-functions of $E$ is needed, since we are not seeking a power-savings in the error term. Arguing as in the proof of Corollary \ref{first corollary}, this implies the bound
$$\underset{\substack{0< |d|\leq D\\ (d,N_E)=1\\ r_{\text{an}}(E_d)\geq 2}}{\sum\nolimits^{*}} 1 =o\left( \frac D {\log \log D}\right).  $$
The same bound holds for the algebraic rank by the work of Gross and Zagier \cite{GZ} and Kolyvagin \cite{Ko}, and applying Lemma \ref{lemma equidistribution root number coprime} we obtain
\begin{equation}
\underset{\substack{0< |d|\leq D\\ (d,N_E)=1\\ r_{\text{al}}(E_d)= 0}}{\sum\nolimits^{*}} 1 = \frac{N(D)}2 +o\left( \frac D {\log \log D}\right), \hspace{1cm} \underset{\substack{0< |d|\leq D\\ (d,N_E)=1\\ r_{\text{al}}(E_d)= 1}}{\sum\nolimits^{*}} 1 = \frac{N(D)}2+o\left( \frac D {\log \log D}\right).
\label{equation equidistribution}
\end{equation}  
 We now apply H\"older's Inequality and Lemma \ref{lemma bound algebraic rank} to bound the average algebraic rank. Selecting $p=(1-\lfloor\log\log\log D\rfloor^{-1})^{-1}>1$ and $q=(1-p^{-1})^{-1}=\lfloor\log\log\log D\rfloor$, we obtain that for any $\epsilon>0$,
\begin{align}
 \underset{\substack{0< |d|\leq D\\ (d,N_E)=1\\ r_{\text{al}}(E_d)\geq 2}}{\sum\nolimits^{*}} r_{\text{al}}(E_d) & \leq \bigg( \underset{\substack{0< |d|\leq D\\ (d,N_E)=1\\ r_{\text{al}}(E_d)\geq 2}}{\sum\nolimits^{*}} 1 \bigg)^{\frac 1p} \bigg( \underset{\substack{0< |d|\leq D\\ (d,N_E)=1\\ r_{\text{al}}(E_d)\geq 2}}{\sum\nolimits^{*}} r_{\text{al}}(E_d)^q \bigg)^{\frac 1q} \notag \\
 & \ll_{\epsilon,E} \bigg(\epsilon \frac D {\log\log D} \bigg)^{\frac 1p} \bigg( \underset{\substack{0< |d|\leq D\\ (d,N_E)=1}}{\sum\nolimits^{*}} (18 +O_E(1))^q\omega(d)^q \bigg)^{\frac 1q}.
 \label{before applying moment bound}
 \end{align}
 
The centered moments of $\omega(n)$ are estimated uniformly in \cite[Theorem 1]{GS}. This translates to a bound on the $ q -$th moment as follows, for $D$ large enough:
\begin{align*}
\sum_{d \leq D} \omega(d)^q &= \sum_{j=0}^q \binom qj (\log\log D)^{q-j}\sum_{d \leq D} (\omega(d)-\log\log D)^j\\
&\leq 2 D\sum_{j=0}^q \binom qj (\log\log D)^{q-\frac j2} \frac{\Gamma(j+1)}{2^{\frac j2} \Gamma(j/2+1)} \\
&\leq 2D (\log\log D)^q \Big(1  + \sum_{j=1}^q \frac{q^j}{j!}(\log\log D)^{-\frac j2} j^{\frac j2+1}  \Big)
\\ &\leq 2D (\log\log D)^q \Big(1  + \frac{q}{(\log\log D)^{\frac 12}}+(e^{q^2 (\log\log D)^{-\frac 12}}-1) \Big)
\\ &= 2D (\log\log D)^q (1+o(1)),
\end{align*}
since $q= \lfloor \log\log\log D \rfloor$.
We obtain that \eqref{before applying moment bound} is 
 \begin{align*}
 & \ll_E \epsilon^{\frac 12} D^{\frac 1p} (\log\log D)^{-\frac 1p} \big(  3 D (\log\log D)^q \big )^{\frac 1q} \leq 3 \epsilon^{\frac 12} D (\log\log D)^{\frac 1{\lfloor\log\log\log D \rfloor}} \leq 4 e\epsilon^{\frac 12} D.
 \end{align*}
Taking $\epsilon$ arbitrarily small and using \eqref{equation equidistribution}, we obtain
$$ \underset{\substack{0< |d|\leq D\\ (d,N_E)=1}}{\sum\nolimits^{*}} r_{\text{al}}(E_d) = \frac{N(D)}2 +o(D)  \sim \frac{N(D)}2,  $$
since $N(D)\asymp_E D$.

\end{proof}

\appendix
\section{The distribution of $S(D;P)$}
In this section we expand on the probabilistic study of $S(D;P)$. Before this, we show that the families of $L$-functions we are considering have a bounded number of repetitions.

\begin{lemma}
\label{lemma bounded number of repetitions}
Fix $E$ an elliptic curve over $\mathbb Q$. There exists an absolute constant $C$ (see \cite[Theorem 5]{M} for the explicit value) such that at most $C$ minimal Weierstrass Equations have the property that the $L$-functions of the associated elliptic curves match that of $E$. 
\end{lemma}
\begin{proof}
Let $y^2=x^3+ax+b$ be the minimal Weierstrass Equation of $E$. It is a well known fact that two elliptic curves with distinct minimal Weierstrass equations are not isomorphic (see for instance \cite[Section III]{Sil}). 

If two elliptic curves $E'$ and $E''$ are isogenous, then their reductions have the same number of points and thus their $L$-functions are identical (see for instance \cite[Theorem 11.67]{Kn}). Conversely, if $L(E',s)=L(E'',s)$, then the reductions of these elliptic curves have the same number of local points, and thus their Frobenius elements have the same characteristic polynomial. It follows that their Tate Modules are isomorphic. The Isogeny Theorem, which follows from Falting's work \cite{Fa}, then implies that $E'$ and $E''$ are isogenous.

It follows that $L(E',s)=L(E'',s)$ if and only if $E'$ and $E''$ are isogenous. However Mazur proved \cite[Theorem 5]{M} that at most $C$ isomorphism classes of elliptic curves are isogenous to a fixed curve $E$, for some absolute constant $C$. This concludes the proof.

\end{proof}

\label{appendix}
We now give a probabilistic argument which supports \eqref{equation strong montgomery}.
We will divide the left hand side of \eqref{equation strong montgomery} by $x^{\frac 12}$ and put $x=e^y$, ending up in the quantity
$$T(D;y) := \sumdw  \sum_{\rho_d \notin \mathbb R} \frac{e^{y(\rho_d- 1/2) }}{\rho_d(\rho_d+1)}.$$
Our goal is to give arguments in favor of the bound $T(D;e^y) \ll_{\epsilon} D^{\frac 12+\epsilon}$.
An important assumption is the following:

\medskip
\noindent \textbf{Hypothesis BM.} The multiset $Z_E=\{ \Im(\rho)\neq 0 : L(\rho,f_{E_d})=0, d\neq 0 \text{ is squarefree} \}$ has bounded multiplicity.\footnote{By this we mean that there exists an absolute constant $C$ such that each element of $E$ has multiplicity at most $C$.}

We will see that under ECRH and BM, $T(D;y)$ has variance $\asymp D\log D$. It follows that $T(D;y)$ is normally $\ll_{\epsilon} D^{\frac 12+\epsilon}$, which justifies \eqref{equation strong montgomery}, as well as Hypothesis M and Hypothesis M$(\delta,\eta)$ for all $0<\eta<\tfrac 12$.

\begin{proposition}
Assume ECRH. Then $T(D;y)$ has a limiting distribution $\mu_D$ as $y\rightarrow \infty$. If we moreover assume BM, then the first two moments of this distribution satisfy
$$ \int_{\mathbb R} t d\mu_D(t) = 0 \hspace{1cm} \int_{\mathbb R} t^2 d \mu_D(t) \asymp_E D\log D.  $$

\end{proposition}

\begin{proof}
Hypothesis ECRH implies that $T(D;y)$ is a Besicovitch $B^2$ almost-periodic function of the form
$$T(D;y) = \sumdw  \sum_{\gamma_d \neq 0} \frac{e^{iy\gamma_d }}{\rho_d(\rho_d+1)}.$$

The existence of the limiting distribution follows from \cite{ANS}. To compute the first two moments of $\mu_D$, we follow the arguments of \cite{Fi}. As in \cite[Lemma 2.5]{Fi}, ECRH implies that for $k\geq 1$,
$$ \int_{\mathbb R} t^k d\mu_D(t) = \lim_{Y \rightarrow \infty} \frac 1Y \int_2^Y T(D;y)^k dy. $$
The claim on the first moment follows from the fact that
$$ \frac 1Y \int_2^Y T(D;y) dy=\frac 1Y\sumdw  \sum_{\gamma_d \neq 0} \frac{e^{iy\gamma_d}}{i\gamma_d\rho_d(\rho_d+1)} \Bigg|_2^Y\ll_{E,D} \frac 1Y. $$
For the variance we will apply Parseval's Identity for Besicovitch almost-periodic functions. Let $W_E$ be the set $Z_E$ without multiplicities. We have 
$$T(D;y) = \sum_{\gamma \in W_E} \frac{e^{iy\gamma }}{(\frac 12+i\gamma)(\frac 32+i\gamma)}c_{\gamma},$$
where $c_{\gamma}$ is defined by the formula
$$ c_{\gamma}:= \underset{\substack{d\neq 0\\ (d,N_E)=1 \\ \exists \gamma_d=\gamma}}{\sum\nolimits^{*}} w\left( \frac dD\right) m_{\gamma_d}, $$
with $ m_{\gamma_d}$ the multiplicity of $\tfrac 12+i\gamma_d$ as a zero of $L(s,f_{E_d})$. 
Parseval's Identity then reads 
$$ V_D:=\int_{\mathbb R} t^2 d\mu_D(t) = \sum_{\gamma \in W_E} \frac{c_{\gamma}^2}{|\frac 12+i\gamma|^2|\frac 32+i\gamma|^2}.
$$
We first give a lower bound on $V_D$. We have:
$$ V_D \gg \sum_{\substack{\gamma \in W_E \\ |\gamma|\leq 1}} c_{\gamma}^2 \geq \sum_{\substack{\gamma \in W_E \\ |\gamma|\leq 1}}\underset{\substack{d\neq 0\\ (d,N_E)=1 \\ \exists \gamma_d=\gamma}}{\sum\nolimits^{*}}  w\left( \frac dD\right)^2 \geq \underset{\substack{d\neq 0\\ (d,N_E)=1 }}{\sum\nolimits^{*}} w\left( \frac dD\right)^2 \sum_{\substack{\gamma_d  \\ |\gamma_d|\leq 1 }} 1 \asymp_E D \log D.  $$
As for the upper bound, we compute using Hypothesis BM:
\begin{align*}
V_D &\ll  \sum_{\gamma \in W_E} \frac{1}{(\gamma+1)^4} \underset{\substack{d_1,d_2\neq 0\\ (d_1d_2,N_E)=1 \\ \exists \gamma_{d_1}=\gamma \\ \exists \gamma_{d_2}=\gamma}}{\sum\nolimits^{*}} w\left( \frac {d_1}D \right)w\left( \frac {d_2}D \right) \\
&\leq \underset{\substack{d_1\neq 0\\ (d_1,N_E)=1 }}{\sum\nolimits^{*}} w\left( \frac {d_1}D \right)\sum_{\substack{\gamma \in W_E \\ \exists \gamma_{d_1}  =\gamma }} \frac{1}{(\gamma+1)^4} \underset{\substack{d_2\neq 0\\ (d_2,N_E)=1 \\ L(E_{d_1},s) \text { and }  L(E_{d_2},s)\\ \text{ have the common zero }\gamma }}{\sum\nolimits^{*}} w\left( \frac {d_2}D \right) \\
& \ll \underset{\substack{d_1\neq 0\\ (d_1,N_E)=1 }}{\sum\nolimits^{*}} w\left( \frac {d_1}D \right)\sum_{\substack{\gamma \in W_E \\ \exists \gamma_{d_1}  =\gamma }} \frac{1}{(\gamma+1)^4} \\
& \ll \underset{\substack{d_1\neq 0\\ (d_1,N_E)=1 }}{\sum\nolimits^{*}} w\left( \frac {d_1}D \right)\sum_{\substack{\gamma_{d_1} }} \frac{1}{(\gamma_{d_1}+1)^4} \asymp_E D\log D,
\end{align*}
from which the result follows.
\end{proof}

\section*{Acknowledgements}
This work was supported by an NSERC Postdoctoral Fellowship, and was accomplished at the University of Michigan. I would like to thank Andrew Granville, Jeffrey Lagarias, James Maynard, James Parks, Kartik Prasanna, Ari Shnidman, Joseph Silverman, Anders S\"odergren and Matthew P. Young for their help and for inspiring conversations.

\end{document}